\documentclass{amsart}
\usepackage{amsmath,amssymb,graphicx,graphics,verbatim,color}

\newtheorem{theorem}{Theorem}[section]
\newtheorem{proposition}[theorem]{Proposition}

\newtheorem{lemma}[theorem]{Lemma}

\newtheorem{remark}[theorem]{Remark}
\newtheorem{definition}[theorem]{Definition}

\newcommand{\dieter}[1]{{\bf [~Dieter:\ }\emph{#1}\textbf{~]}}

\begin{document}
\title{On rigidity, orientability and cores of random graphs with sliders}
\author{J.~Barr\'e$^1$,\\
M.~Lelarge$^2$,\\ 
D.~Mitsche$^{1}$} 
\footnotetext[1]{Laboratoire J.~A.~Dieudonn\'e, UMR CNRS 7351,
Universit\'e de Nice-Sophia Antipolis, Parc Valrose, 06108 Nice France, email: \{julien.barre, dmitsche\}@unice.fr }
\footnotetext[2]{INRIA-ENS, Paris, France, email: marc.lelarge@inria.fr}
\date{}
\thanks{Julien Barr\'{e} thanks Victor Mizrahi and Alain Olivetti for many discussions.}
\maketitle
\begin{abstract}
  Suppose that you add rigid bars between points in the plane, and
  suppose that a constant fraction $q$ of the points moves freely in the whole plane; the remaining fraction is constrained to move 
  on fixed lines called sliders. When does a giant rigid cluster
  emerge?  Under a genericity condition, the answer only depends on
  the graph formed by the points (vertices) and the bars (edges). We find for the random graph $G \in \mathcal{G}(n,c/n)$ the threshold
  value of $c$ for the appearance of a linear-sized rigid component as
  a function of $q$, generalizing results of~\cite{Moore}. We show
  that this appearance of a giant component undergoes a continuous
  transition for $q \leq 1/2$ and a discontinuous transition for $q >
  1/2$. In our proofs, we introduce a generalized notion of orientability interpolating between 1- and 2-orientability, 
  of cores interpolating between $2$-core and $3$-core, and of extended cores interpolating between $2+1$-core
  and $3+2$-core; we find the precise expressions for the respective thresholds and the sizes of the different cores above the threshold.
   In particular, this proves a
  conjecture of~\cite{Moore} about the size of the $3+2$-core. We also derive some structural properties of rigidity with sliders (matroid and decomposition into components) which can be of independent interest.

\end{abstract}

\section{Introduction}
\label{sec:intro}
Consider a set of points, some of them allowed to move freely in the
Euclidean plane, and some constrained to move on fixed lines, called
sliders. The free points have two degrees of freedom, the points
attached to sliders have only one. Now, add bars between pairs of
these points; a bar fixes the length between the two end-points.  The
points and bars form a \emph{framework}. A framework is said to be
rigid if it cannot be deformed (but can possibly be translated and
rotated on the plane); equivalently, it is rigid if the distance
between any pair of points, connected by a bar or not, is fixed.
Characterizing the rigidity of a framework is very difficult in
general. In the absence of sliders, a celebrated theorem by
Laman~\cite{Laman} ensures that for a generic framework, its rigidity
properties only depend on its underlying graph, where points are
vertices and bars are edges: the geometry does not enter. This theorem
has been generalized to frameworks with sliders in~\cite{Streinu}.  In
the whole article, we will implicitly assume that all frameworks are
generic, so that rigidity has a purely graph theoretical
characterization and we can deal with vertices and edges instead of
points and bars. 
This will be detailed in Section~\ref{sec:defs}.
  
We will call a vertex of \emph{type 1} (resp. \emph{type 2}) if it is
(resp. is not) connected to a slider. Consider now a percolation
setting: take a set of $n$ vertices, a fraction $q$ of which are type 2,
and add edges randomly. The questions are: When does a giant (that is:
including a positive fraction of the vertices) rigid structure emerge?
What is its size?  When edges are sampled independently at random between
pairs of vertices, the resulting graph is an Erd\H{o}s-R\'enyi random
graph $G(n,c/n)$. In this case and for $q=1$ (no slider),
Kasiviswanathan et al. \cite{Moore} showed that the threshold for
a giant rigid component is $c\simeq 3.588$, and that the transition is
discontinuous: as soon as the giant rigid component appears, it already 
includes a positive fraction of all $n$ vertices. This recovers numerical
and heuristic results found earlier in the physics literature
\cite{Moukarzel97,Moukarzel99}, and contrasts with the emergence
of a giant connected component at $c=1$, which is continuous. 

Indeed when $q=0$, we will see that rigidity is closely related to the emergence of the giant connected component.
Our goal is to investigate the case where $q \in [0,1]$. 
We are thus interested in situations interpolating between
standard connectivity percolation and rigidity percolation as studied
in \cite{Moore}.  
We obtain the following results:
\begin{itemize}
\item We compute the threshold for rigidity percolation as a function of $q$
\item We show that the transition is continuous for $q \leq 1/2$ and
  discontinuous for $q>1/2$, thus uncovering what is called a
  "tricritical" point in statistical mechanics, for $q=1/2$
\item On the way, we obtain new results on cores for Erd\H{o}s-R\'enyi
  random graphs and their generalization to two types of
  vertices. We prove in particular a conjecture on the size of the 
  $3+2$-core in~\cite{Moore}
\end{itemize}

Rigidity percolation has physical motivations: it is a model to
understand some properties of network glasses and proteins
\cite{Thorpe1983,Phillips1985,Boolchand1986,Thorpe_proteins}. Thus, problems
related to ours have been investigated by theoretical physicists. We
have already cited investigations on random graphs starting with
\cite{Moukarzel97,Moukarzel99}, with only one type of vertex (type 2,
or more generally type $k$). In \cite{Moukarzel2003}, Moukarzel
heuristically studies a model with two types of vertices: a
fraction of the vertices are pinned to the plane, instead of being
allowed to move in one direction; they could be called ``type 0''
vertices. In this case, the transition disappears when the fraction of
pinned vertices increases: there is no tricritical point, but rather a
critical point. 

In order to compute the threshold for rigidity, we use the same
connection as~\cite{Moore} between orientability and rigidity. We then
use recently introduced and powerful methods to compute the
orientability threshold \cite{Lelarge}. To investigate the continuous or
discontinuous character of the transition, we rely on various
refinements of a method introduced in~\cite{Janson} to investigate the cores
of a random graph.
In Section \ref{sec:defs}, we define our notion of rigidity with sliders and state our main results for Erd\H{o}s-R\'enyi random graphs. In Section \ref{sec:det}, we gather our structural results for rigidity with sliders: matroid and decomposition into components. We then prove our results for random graphs: in Section \ref{sec:orient}, we compute the orientability threshold, in Section \ref{sec:or-rig}, we relate it to rigidity. We then prove our main Theorems in Sections \ref{sec:theo_cont}, \ref{sec:theo_core} and \ref{sec:theo_core+}. Finally a technical but important Lemma is proved in Section \ref{sec:lemma_important}.

\section{Some definitions on rigidity and statements of  results}
\label{sec:defs}

Throughout this paper $\log$ denotes the natural logarithm. Also, throughout the paper $G$ is a graph $(V,E)$ with $|V|=n$ and $|E|=m$. All
our graphs are simple. Vertices are either of type 1 or of type 2, and for
$i \in \{1,2\}$, $n_i$ denotes the number of vertices of type $i$, so that
$n=n_1+n_2$.

Subgraphs are typically denoted by $G'$ with $n_i(G')$ vertices of type $i\in
\{1,2\}$, $n(G')=n_1(G')+n_2(G')$ vertices in total and $m(G')$ edges. When
the context is clear, we use the notations: $n'=n(G')$, $n_i'=n_i(G')$ and
$m'=m(G')$.

\begin{definition} 
  Let $G$ be a graph with $n=n_1+n_2$ vertices and $m$ edges. $G$ is
  \textbf{sparse} if for all subgraph $G'\subseteq G$ on
  $n'=n'_1+n'_2\geq 2$ vertices and $m'$ edges, we have:
\[
m'\leq n'_1+2n'_2+\min(0,n'_1-3) = 2n'-\max(n'_1,3).
\] 
\end{definition}

In terms of physics, a sparse graph represents a structure without
redundant constraint. The special treatment needed for subgraphs with
0, 1 or 2 vertices of type $1$, i.e. when $n'_1< 3$, can then be
understood: a structure which is not connected at all to the
underlying plane (that is $n'_1=0$) cannot be pinned, and always keeps
at least three degrees of freedom, hence the $-3$; a structure with
one slider (that is $n'_1=1$) always keeps at least two degrees of
freedom, hence the $-2$; and similarly for $n'_1=2$. If $n'_1\geq 3$,
the structure can be completely pinned to the underlying plane, and
thus has zero degrees of freedom.

\begin{remark}\label{rem:Streinu1} We follow here Streinu and Theran \cite{Streinu},
    with a simplified terminology to make the present article easier
    to read. The present definition of sparsity corresponds to their
    $(2,0,3)$-graded-sparsity, for a restricted class of graphs (they 
consider also multiple graphs, and more types of vertices). Since
    we are only using two concepts of sparsity (see definition of
    Laman-sparsity below), no confusion should arise. To make
    the connection more explicit, note that our ``type 1 vertices'' correspond to
    vertices ``with one attached loop'' in \cite{Streinu}.
\end{remark}

We recall the standard definition:
\begin{definition}
  $G$ is Laman-sparse if for all subgraph $G'\subseteq G$ with $n'
  \geq 2$, $m'\leq 2n'-3$.
\end{definition}
Laman-sparsity and sparsity are equivalent if there are only vertices of
type $2$, i.e. $n=n_2$. Moreover a sparse graph is always Laman-sparse.
\begin{definition}
$G$ is \textbf{minimally rigid} if either $n=1$, or $G$ is sparse
and
\begin{eqnarray}
m= n_1+2n_2+\min(0,n_1-3).
\end{eqnarray}
\label{def_rigid}
\end{definition}

\begin{lemma}
If $G$ is minimally rigid with $n_1<6$,
then $G$ is connected.
\end{lemma}
\begin{proof}
Consider a partition of the vertices in two parts with $n_a$ and $n_b$
vertices respectively. Let $m_a$ and $m_b$
be the number of edges induced by each part. By the sparsity of $G$,
we have $m_i \leq 2n_i-3$ for $i\in \{a,b\}$.
Hence, we have
\begin{eqnarray*}
m - (m_a+m_b) \geq 2n-\max(3,n_1) - 2n_a-2n_b+6 = 6-\max(3,n_1),
\end{eqnarray*}
so that for $n_1<6$, the two parts are connected.
\end{proof}

\begin{remark}
For $n_1\geq 6$, a minimally rigid graph $G$ does not need to be connected as seen by considering the disjoint union of two cliques of size three with all nodes of type $1$.
\end{remark}
\begin{remark} 

Streinu and Theran (see~\cite{Streinu}) use a slightly different
  definition of rigidity. In our notation,
  for them $G$ is minimally rigid if $G$ is sparse and
  $m=n_1+2n_2$. This definition is not equivalent to ours: using our
    Definition~\ref{def_rigid}, physically, it means that we consider as
  rigid a structure that cannot be deformed (but can possibly be moved
  over the plane as a solid object). Streinu and Theran
  (see~\cite{Streinu}) consider as rigid a structure that cannot
  be deformed, and that is pinned on the plane; in particular,
  rigidity in this sense implies $n_1\geq
  3$. Definition~\ref{def_rigid}, however, coincides with the standard
  definition of rigidity when there are only vertices of type $2$,
  and this will be convenient to compare our results with the results
  of~\cite{Moore}; it will also allow us to use some of their results.
\end{remark}

Recall that a spanning subgraph is one that includes the entire vertex set $V$.
\begin{definition}
  A graph is \textbf{rigid} if it contains a spanning subgraph which
  is minimally rigid. A \textbf{rigid block} in $G$ is defined to be a
  vertex-induced rigid subgraph. A \textbf{rigid component} of
  $G$ is an inclusion-wise maximal block. 
\end{definition}
\begin{remark}
Note that for a sparse graph $G$, a rigid block is always minimally rigid.
\end{remark}
Note that a rigid component does not need to be connected.
By definition, it is clear that rigidity is preserved under addition
of edges and that the size of the largest (in terms of vertices
covered) rigid component of a graph can only increase when edges are
added.

We now describe our probabilistic setting: consider for the following
statements the random graph $G \in \mathcal{G}(n,c/n)$ where each edge
is present independently with probability $c/n$, with $c > 0$. For
such a graph we also write $G(n,c/n)$ below. Each vertex gets type $1$
with probability $1-q$ and type $2$ with probability $q$, where $q \in
[0,1]$.

To state our result, we need some notations. Let $Q(x,y) =
e^{-x}\sum_{j\geq y}\frac{x^j}{j!}$. We define the function $c^*(q)$
as follows: 
\begin{itemize}
\item for $q\leq 1/2$, we set $c^*(q)=\frac{1}{1-q}$;
\item for $q>1/2$, let
$\xi^*=\xi^*(q)$ be the positive solution to:
\begin{eqnarray*}
\xi \frac{(1-q) Q(\xi,1)+qQ(\xi,2)}{(1-q) Q(\xi,2)+ 2qQ(\xi,3)}=2.
\end{eqnarray*}
In this case we set:
\begin{eqnarray*}
c^*(q) = \frac{\xi^*}{(1-q) Q(\xi^*,1)+qQ(\xi^*,2)}.
\end{eqnarray*}
\end{itemize}

It will follow from the proof that the equation for $\xi^*$ has
indeed a unique positive solution and that for $q > 1/2$, $c^*(q) <
\frac{1}{1-q}$.

We can now state our first theorem:
\begin{theorem}\label{theo_threshold}
  Let $G=G(n,c/n)$ with $c > 0$, and let $q \in [0,1]$. Let $R_n(q,c)$ ($R_n^C(q,c)$, resp.)
  be the number of vertices covered by the largest rigid component
  (connected rigid block, resp.) of $G$.
\begin{itemize}
\item For $c>c^*(q)$, there is a
  giant rigid component in $G$, i.e., there exists
  $\alpha=\alpha(q,c)>0$ such that
\[
\mathbb{P}\left(\frac{R_n(q,c)}{n}\geq \alpha\right) \to
1~\mbox{when}~n\to \infty
\]
\item For $c<c^*(q)$, there is no giant
rigid component in $G$; i.e.,
\[
\forall \alpha>0~,~ \mathbb{P}\left(\frac{R_n(q,c)}{n}\geq \alpha\right) \to
0~\mbox{when}~n\to \infty
\]
\end{itemize}
The above results also hold true for $R_n^C(q,c)$.  
  Moreover, for $c > c^*(q)$, a.a.s., there is one unique giant
rigid component (one unique giant connected rigid block, resp.).
\end{theorem}

Our next theorem states that the transition as $c$ varies and $q$ is
held fixed is continuous for $q \leq
1/2$ and discontinuous for $q > 1/2$. More precisely, we have the
following:
\begin{theorem}\label{theo_continuous}
  $\bullet$ The transition is discontinuous for $q>1/2$: let $q>1/2$;
  there is $\alpha(q)=\alpha>0$ such that for any $c>c^\ast(q)$
\[
\lim_{n\to \infty} \mathbb{P}\left(  \frac{R_n(q,c)}{n}\geq \alpha \right) =1
\]
$\bullet$ The transition is continuous for $q \leq 1/2$: let $q \leq
1/2$; for any $\alpha>0$,
\[
\lim_{c\to \frac{1}{1-q}} \lim_{n\to \infty} \mathbb{P}\left(
  \frac{R_n(q,c)}{n}\geq \alpha \right) = 0
\]
\end{theorem}

We now relate rigidity and orientability. We start with the following
definition of \textbf{$1.5$-orientability}.
\begin{definition}
  A graph is 1.5-orientable if there exists an orientation of the
  edges such that type 1 vertices have in-degree at most 1 and type 2
  vertices have in-degree at most 2.
\end{definition}

A standard argument in the context of network flows gives (see
Proposition 3.3 in \cite{sparsity}) 
\begin{proposition}\label{prop:flow}
A graph $G$ is $1.5$-orientable if and only if for every induced
subgraph $G'$ of $G$, $m'\leq n'_1+2n'_2$.
\end{proposition}
As a corollary, we see that a sparse graph $G$ is always
$1.5$-orientable. Moreover, we see that if $G$ is $1.5$-orientable,
then $G$ will remain $1.5$-orientable after removing some edges and if
$G$ is not $1.5$-orientable then adding edges cannot make it
$1.5$-orientable.
 
Our next theorem shows that the threshold for being $1.5$-orientable
for the random graph $G(n,c/n)$ is the same as the one for the
appearance of a giant rigid component.
\begin{theorem}\label{prop:orientabilite}
Let $G=G(n,c/n)$ with $c > 0$, and let $q \in [0,1]$.
\begin{itemize}
\item[(a)] if $c<c^*(q)$, $G$ is 1.5-orientable a.a.s.
\item[(b)] if $c>c^*(q)$, $G$ is not 1.5-orientable a.a.s.
\end{itemize}
\end{theorem}

We now relate the notion of rigidity and $1.5$-orientability with a
new notion of core.
\begin{definition}
For a graph with type $1$ and type $2$ vertices, the \textbf{2.5-core}
is the largest induced subgraph with all type $1$ vertices with degree
at least $2$ and all type $2$ vertices with degree at least $3$.
\end{definition}
Note that this definition coincides with the 2-core (3-core, resp.) if
the graph contains only type 1 vertices (type 2, resp.).

One can show that we can construct the $2.5$-core by removing
recursively type $1$ vertices with degree at most $1$ and type $2$
vertices with degree at most $2$.
Note that the $2.5$-core can be empty and in this case, the graph is
$1.5$-orientable. More generally, a graph $G$ is $1.5$-orientable if and only
if its $2.5$-core is orientable.

Clearly the size of the $2.5$-core can only increase with the addition
of edges.
In our probabilistic setting, it turns out that for a fixed $q$, the
$2.5$-core appears at a value $\tilde{c}(q)\leq c^*(q)$.

Let
$Q(x,y)$ as before. We define
\begin{eqnarray}
\label{def:tildec}\tilde{c}(q) = \inf_{\xi>0} \frac{\xi}{(1-q)Q(\xi,1)+qQ(\xi,2)}.
\end{eqnarray}
Note that when $\xi\to 0$, we have $\frac{\xi}{(1-q)Q(\xi,1)+qQ(\xi,2)} \to \frac{1}{1-q}$, in particular $\tilde{c}(q) \leq \frac{1}{1-q}$.
Let $\tilde{\xi}(q,c)$ be the largest solution to
\begin{eqnarray}
\label{def:tildexi}\xi=c(1-q)Q(\xi,1)+cqQ(\xi,2).
\end{eqnarray} We can now state the theorem:
\begin{theorem}\label{theo_core}
  Let $G=G(n,c/n)$ with $c > 0$ and let $q \in [0,1]$. Let $Core$ be the 2.5-core of
  $G$, $n_1(Core)$ ($n_2(Core)$, resp.) be the number of nodes
  of type 1 (type 2, resp.) in the core and $m(Core)$ be the number of edges
  in the core.  We have
\begin{itemize}
\item[(a)] if $c<\tilde{c}(q)$ and $q>0$, then a.a.s. the 2.5-core has
  $o_p(n)$ vertices.
\item[(b)] if $c>\tilde{c}(q)$, then a.a.s. the 2.5-core is such that $n_1(Core)/n \to (1-q)Q(\tilde{\xi}(q,c)),2)$, $n_2(Core)/n \to
  qQ(\tilde{\xi}(q,c)),3)$,\\ and $2m(Core)/n \to
  \tilde{\xi}(q,c)\left((1-q)Q(\tilde{\xi}(q,c)),1)+qQ(\tilde{\xi}(q,c)),2)\right)$.
\end{itemize}
\end{theorem}

\begin{remark}
When the core is not $o_p(n)$, i.e., when $c>\tilde{c}(q)$, we have
\begin{eqnarray*}
  \frac{m(Core)}{n_1(Core) +2n_2(Core)} \to \frac{\tilde{\xi}(q)}{2} \frac{(1-q)Q(\tilde{\xi}(q)),1)+qQ(\tilde{\xi}(q)),2)}{(1-q)Q(\tilde{\xi}(q)),2)+2qQ(\tilde{\xi}(q)),3)}.
\end{eqnarray*}
In particular, if this ratio is larger than one, then the 2.5-core is
not 1.5-orientable. 
A simple computation shows that this ratio becomes larger than one
exactly for $c>c^*(q)\geq \tilde{c}(q)$.
Moreover, we have $c^*(q) = \tilde{c}(q)=\frac{1}{1-q}$ for $q\leq 0.5$ and $c^*(q)
> \tilde{c}(q)$ for $q>0.5$.
\end{remark}

\begin{remark}
  When $q$ is fixed and we increase $c$ from 0 to infinity, it is easy
  to note the following from previous theorem: when $q\leq 1/2$, the
  size of the 2.5-core is continuous in $c$ whereas for $q>1/2$, the
  2.5-core appears discontinuously.
\end{remark}


In the absence of sliders ($q=0$), the largest rigid
  component is closely related to the $3+2$-core \cite{Moore}.  This
  led the authors of \cite{Moore} to formulate a conjecture on the
  size of the $3+2$-core. We introduce now a generalization of the
  $3+2$-core which will play a role in our proof of
  Theorem \ref{theo_continuous}.

\begin{definition}
  Starting from the 2.5-core, one constructs a larger subgraph as
  follows: add recursively type 1 vertices which are linked by one
  edge with the current subgraph, and type 2 vertices which are linked
  by two edges with the current subgraph. The resulting subgraph is
  called the \textbf{$2.5+1.5$-core}.
\end{definition}
Note that this definition coincides with the 2+1-core (3+2-core, resp.)
if the graph contains only type 1 vertices (type 2, resp.).

Furthermore, we also compute the threshold and the size of the
2.5+1.5-core. This proves a conjecture in \cite{Moore} on the
3+2-core. The proof follows again the ideas in~\cite{Janson}.  We use
the same definitions of $\tilde{c}(q)$ and $\tilde{\xi}(q)$ as before
and state the following theorem:
\begin{theorem}\label{theo_core+}
Let $G=G(n,c/n)$ with $c > 0$ and $q \in [0,1]$. Let $Core+$ be the $2.5+1.5$-core of
  $G$, and $n(Core+)$ the number of vertices inside the
  $2.5+1.5$-core.  If $c>\tilde{c}(q)$, where $\tilde{c}(q)$ is defined by (\ref{def:tildec}), then a.a.s., $n(Core+)/n\to 1-e^{-\tilde{\xi}}-q \tilde{\xi} e^{-\tilde{\xi}}$, where $\tilde{\xi}$ is defined in (\ref{def:tildexi}).
\end{theorem}

\begin{remark}~\label{rem_core+} For $q \leq 1/2$, we have 
  $\tilde{c}(q)=\frac{1}{1-q}$, and 
 if $c \searrow \frac{1}{1-q}$, then we have $\tilde{\xi} \to 0$, and thus
  $n(Core+)/n \to 0$.
\end{remark}

For the proof of the aforementioned theorems, the following lemma
plays a crucial role, and hence we state it already here: for a
subgraph of size $n'$, let $n'_1$ its number of vertices of type $1$ and
$n'_2$ its number of vertices of type $2$ (we do not explicitly refer to
the size nor to the subgraph, since it is clear from the context).
Let $X_{n'}$ denote the number of subgraphs of size
$n'$ with more than $n'_1+2n'_2$ edges. We have:
\begin{lemma} \label{lemma_important} Let $q \in (0,1)$, and let $G
  \in \mathcal{G}(n,p)$ with $p=c/n$ and $c < \frac{1}{1-q}$.  A.a.s., there
  exists a strictly positive constant $\alpha=\alpha(q, c -
  \frac{1}{1-q}) > 0$ such that $\sum_{1 \leq n' \leq \alpha n} X_{n'}=0$.
\end{lemma}
\begin{remark} Lemma~\ref{lemma_important} will also play the role of
  the Lemma~4.1 in~\cite{GaoWormald}, or Proposition 3.3
  in~\cite{Moore}. Lemma~4.1 in~\cite{GaoWormald} ensures that all
  subgraphs of size $u$, with $m$ edges, such that $m/u>c_1>1$ are of
  size at least $\gamma n$ for some $\gamma>0$. In our case however,
  if $n_2$, the number of type 2 sites is much smaller than $n_1$,
  this lemma cannot be used. Lemma~\ref{lemma_important} provides the
  necessary refinement.
\end{remark}
\bigskip

\section{Properties of (deterministic) sparse graphs}\label{sec:det}

We gather in this section a few properties valid for general graphs,
independently of the probabilistic setting. They will be useful later.

Given two subgraphs $A=(V_A,E_A)$ and $B=(V_B,E_B)$ of $G$, we denote
by $A\cup B$ ($A\cap B$, resp.) the subgraph of $G$ with vertex set
$V_A \cup V_B$ ($V_A \cap V_B$, resp.) and edge set $E_A\cup E_B$
($E_A \cap E_B$, resp.).

\begin{lemma}\label{lem:intuni}
Given two rigid blocks $A=(V_A,E_A)$ and $B=(V_B,E_B)$ of a sparse
graph $G$, we have
\begin{itemize}
\item if $n(A\cap B)\geq 2$, then $A\cup B$ and $A\cap B$ are rigid
  blocks;
\item if $n(A\cap B)\geq 1$ and $\min(n_1(A),n_1(B))\geq 3$, then $n_1(A\cap B)\geq
  3$ and in particular $A\cup B$ and $A\cap B$ are rigid blocks.
\end{itemize} 
\end{lemma}
\begin{proof}
We first note that for any $x,y,z\geq 0$, such that $\min(x,y)\geq z$, we have
\begin{eqnarray}
\label{eq:maxunin}\max(x+y-z,3)+\max(z,3) \geq \max(x,3)+\max(y,3).
\end{eqnarray}
Denoting by $m(\Delta)$ the number of edges between $V_A\setminus V_B$ and
  $V_B\setminus V_A$, we have
\begin{eqnarray*}
m(A\cup B) &=& m(A) + m(B) - m(A\cap B) +m(\Delta)\\
&=& 2n(A)-\max(n_1(A),3)+2n(B)-\max(n_1(B),3)-m(A\cap B)+m(\Delta)
\end{eqnarray*}
By the sparsity of $G$, we have $m(A\cup B)\leq 2n(A\cup B)
-\max(n_1(A\cup B),3)$, so that we get
\begin{eqnarray*}
m(A\cap B) \geq 2n(A\cap B)-\max(n_1(A),3)-\max(n_1(B),3)+\max(n_1(A\cup B),3)
 +m(\Delta).
\end{eqnarray*}
Using (\ref{eq:maxunin}), we get
\begin{eqnarray*}
m(A\cap B) \geq 2n(A\cap B)-\max(n_1(A\cap B),3) +m(\Delta).
\end{eqnarray*}
First assume that $n(A\cap B)\geq 2$, so that by sparsity of $G$, we
get $m(\Delta)=0$ and
\begin{eqnarray*}
m(A\cap B) = 2n(A\cap B)-\max(n_1(A\cap B),3).
\end{eqnarray*}
Hence, we have
\begin{eqnarray*}
m(A\cup B) &=& 2n(A)-\max(n_1(A),3)+2n(B)-\max(n_1(B),3)\\
&&\quad-2n(A\cap
B)+\max(n_1(A\cap B),3) \\
&\geq& 2n(A\cup B) - \max(n_1(A\cup B),3),
\end{eqnarray*}
so that by sparsity of $G$, we have indeed an equality and we proved
the first point.

We now assume that $n(A\cap B)\geq 1$ and $\min(n_1(A),n_1(B))\geq 3$,
so that we have $m(A\cup B)\leq 2n_2(A\cup B) +n_1(A\cup B)$ and then
\begin{eqnarray*}
m(A\cap B) \geq 2n_2(A\cap B)+n_1(A\cap B)+m(\Delta).
\end{eqnarray*}
We see that $m(A\cap B)\geq 1$ and hence
$n(A\cap B)\geq 2$. So, again by sparsity of $G$, we get
\begin{eqnarray*}
m(A\cap B)\leq 2n_2(A\cap B)+n_1(A\cap B)+\min\left(0,n_1(A\cap B)-3\right).
\end{eqnarray*}
In particular, we have $n_1(A\cap B)\geq 3$ and then the second point
follows from the first one.
\end{proof}

Next, we show that by changing one vertex from type $2$ to type
$1$, a rigid graph remains rigid. This is the content of the following
lemma:

\begin{lemma}\label{lem:changementtype}
Let $G$ be a minimally rigid graph, and let $v$ be a type 2 vertex.
Define $\tilde{G}$ as the same as $G$ where $v$ is transformed into 
a type 1 vertex. Then $\tilde{G}$ is rigid.
\end{lemma}
\begin{proof}
  Assume first that $n_1(G)<3$. Then $\tilde{G}$ is actually even
  minimally rigid, and the statement follows: indeed, consider a
  subgraph $\tilde{H}$ of $\tilde{G}$. If $v\notin \tilde{H}$, the
  sparsity condition for $\tilde{H}$ is directly inherited from the
  sparsity of $G$. If $v\in \tilde{H}$, consider $H$ the subgraph of
  $G$ with the same vertices as $\tilde{H}$, except that $v$ is type
  2. Then from the sparsity of $G$, we have
\[
m(H) \leq 2n(H)-3,~\mbox{and}~m(H)\leq n_1(H)+2n_2(H).
\]
Since $n_1(H)<3$, it is enough to consider only the condition $m(H) \leq
2n(H)-3.$ Since $m(\tilde{H})=m(H)$, $n(\tilde{H})=n(H)$, the condition
$m(\tilde{H}) \leq 2n(\tilde{H})-3$ is true, and since
$n_1(\tilde{H})\leq 3$, this condition is enough to ensure that
$\tilde{H}$ verifies the sparsity condition. Hence  $\tilde{G}$ is sparse.
It is also clear that $\tilde{G}$ has exactly the right number of edges (we 
use again here $n_1(\tilde{G})\leq 3$). Thus $\tilde{G}$ is rigid.

Assume now $n_1(G)\geq 3$. Now $\tilde{G}$ cannot be minimally rigid
since it has one excess edge. We have to remove an edge, and the
difficulty is to remove the right one.
Define $H$ to be the smallest subgraph of $G$ such that:
\begin{itemize}
\item $H$ contains  $v$
\item $n_1(H)\geq 3$ 
\item $H$ is minimally rigid
\end{itemize}
$H$ exists since $G$ itself verifies all three conditions above.  By
Lemma~\ref{lem:intuni}, $H$ is unique and can be defined as the
intersection of all subgraphs of $G$ verifying the above
conditions. 
In particular, for any $K
\subseteq G$ satisfying the above conditions, $H \subseteq K$.
  
Choose now $e$ any edge in $H$ 
and define
$\bar{G}$ as follows:
\[
\bar{G}= \tilde{G} \backslash \{ e \} 
\]
We prove now that $\bar{G}$ is minimally rigid, which is clearly
enough to show that $G$ is rigid: first, notice that $n_1(\bar{G})=
n_1(G)+1\geq 3$; with respect to $G$, one edge is removed and one
vertex is turned from type 2 to type 1, hence the total number of
edges is correct. 
It thus remains to prove that $\bar{G}$ is
sparse. Assume $\bar{G}$ is not sparse, and take a subgraph $K$ of
$\bar{G}$ violating the sparsity condition. Note that if before
changing $v$ from type $2$ to type $1$, we have $n_1(K) < 3$, by the
argument at the beginning of the lemma, $K$ remains sparse, so we may
assume that already in $G$ we have $n_1(K) \geq 3$.
We have a few cases:\\
{\bf Case 1:} If $v\notin K$, $K$ can be seen as a subgraph of $G$;
the sparsity of $G$ implies that the sparsity condition for $K$ is
true. This contradicts the hypothesis on $K$.  {\bf Case 2:} Assume
now $v \in K$. We have, by assumption on $K$ not being sparse,
\[
m(K) = n_1(K) +2n_2(K)+1.
\] {\bf Case 2a:} Assume $H\subseteq K$. Let $K'$ be the subgraph
corresponding to the vertex set $K$ in $G$. $K'$ had at most
$n_1(K')+2n_2(K')$ edges. Now, $v$ changed its type, but also one edge
$e \in H$ has been removed, hence we have $m(K) \leq n_1(K)+2n_2(K)$,
and $K$ cannot violate
sparsity.\\
{\bf Case 2b:} Assume $H\nsubseteq K$. Define now $K'$ to be equal to
$K$, but turning vertex $v$ from type 1 back to type 2. $K'$ is a
subgraph of $G$. $G$ is sparse, hence $K'$ is sparse.  Since $n_1(K')
\geq 3$ and since $m(K) = n_1(K) +2n_2(K)+1$, we have
\[
m(K') =n_1(K')+2n_2(K').
\]
Hence $K'$ is minimally rigid (in $G$). Also, $v \in K'$, hence $K'$
satisfies all properties defining $H$, and thus, by minimality $H
\subseteq K'$.  This implies in turn that in $\tilde{G}$, $H \subseteq
K$, which is a contradiction, finishing the proof.
\end{proof}

The following proposition is closely related to the concept of ``graded-sparsity
matroids'' introduced in \cite{LeeStreinu}. It shows that the matroid
structure is retained within our slightly modified definitions.
\begin{proposition}\label{prop:matroid}
The collection of all minimally rigid graphs on $n_1$ vertices of type
$1$ and $n_2$ vertices of type $2$ is the set of bases of a matroid
whose ground set is the set of edges of the complete graph on
$n=n_1+n_2\geq 2$ vertices.
\end{proposition}
\begin{proof}
  The case $n_1=0$ is well-known (see \cite{lest08}), so we consider only the case $n_1\geq 1$.

  We first construct a minimally rigid graph. Consider the case where
  $n_1\geq 2$. Start from one cycle with the vertices of type $1$ and
  one cycle with the vertices of type $2$. If there are only two
  vertices of a given type, then the cycle is simply an edge between
  these two vertices and if there is only one vertex of a given type,
  the cycle is empty.  Hence if $n_i<3$ for $i\in \{1,2\}$, the
  corresponding cycle contains $n_i-1$ edges and if $n_i\geq 3$, the
  corresponding cycle contains $n_i$ edges. In particular, since we
  assumed that $n_1\geq 2$, the cycle with vertices of type $1$ has
  $n_1+\min(n_1-3,0)$ edges.  If $n_2=0$, we are done.  If $n_2\geq1$,
  we select one vertex of type $1$ (denoted by $u$) and add an edge
  between this vertex and each vertex of type $2$ to get a minimally
  rigid graph. Then for $n_2\geq 3$ we are done as the graph has
  $2n_2+n_1+\min(0,n_1-3)$ edges and is sparse.  For $n_2=1,2$, we
  need to add an edge, and for example we can add one edge between a vertex of type
  $1$ different from $u$ and any vertex of type $2$.

Consider now the case $n_1=1$. The cases $n_2=1,2$ are easy, just take
the complete graph. For $n_2\geq 3$, start as above with a cycle with
vertices of type $2$ and then add an edge between all vertices of type
$2$ except one and the vertex of type $1$.

We now prove the basis exchange axiom. Let $B_i=(V,E_i)$, $i=1,2$ be
two minimally rigid graphs and $e_2\in E_2\backslash E_1$. We must
show that there exists an edge $e_1\in E_1\backslash E_2$ such that
$(V,E_1\backslash\{e_1\}\cup\{e_2\})$ is minimally rigid.  Let
$e_2=uv$. Consider all the rigid blocks of $B_1$ containing vertices
$u$ and $v$. By Lemma \ref{lem:intuni}, the intersection of these
blocks denoted by $B'=(V',E')$ is still a rigid block of $B_1$. $B'$
is not a rigid block of $B_2$, since otherwise the subgraph $E'\cup
\{e_2\}\subset E_2$ would violate the sparsity of $B_2$ (note that
$u,v\in V'$).  Hence there exists $e_1\in E'\backslash E_2$. We are
done if we prove that $B_3=(V,E_1\backslash\{e_1\}\cup\{e_2\})$
is sparse.
Consider any subgraph $H$ of $B_1$ such that sparsity
is violated in $B_1\cup\{e_2\}$. Note that $H$ is a rigid block of
$B_1$ containing both $u$ and $v$. Since $B'$ is the minimal subgraph
of $B_1$ with this property, $B' \subseteq H$, and then both endpoints
corresponding to $e_1$ are in $H$. The addition of $e_2$ violates
sparsity, but the removal of $e_1$ restores the count, and we are
done.
\end{proof}

\begin{remark}\label{rem:Theran}(due to L. Theran)
As pointed out in Remark~\ref{rem:Streinu1}, Proposition~\ref{prop:matroid} can also be deduced from the fact that, using the terminology of~\cite{Streinu}, all $(2,0,3)$-graded-sparse graphs form a matroid whose ground set is the set of edges of the complete graph together with two loops at each vertex (see~\cite{Streinu}). More precisely, let  $M_1$ be this matroid with ground set $E_1:=\{ E(K_n) \cup \mbox{ 2 loops per vertex} \}$ with independent sets $I_1$, and let $n_1$ and $n_2$ be the number of vertices of type 1 (type 2, resp.). Let $L$ be the set of edges containing exactly $1$ self-loop at each of the $n_1$ vertices of type $1$. Consider then $E_2:=\{E(K_n) \cup L\}$ and note that $E_2 \subseteq E_1$. Moreover,  $I_2 := \{ A \in I_1: A \subseteq E_2\}$ is still a matroid $M_2$, since this corresponds to a truncation of $M_1$. Finally, consider the sets $I_3:= \{ A \in I_2: A \cup L \mbox{ is independent in } M_2 \}$. Since this corresponds to a contraction of $M_2$, the resulting structure is a matroid, that corresponds exactly to the matroid described in Proposition~\ref{prop:matroid} (all elements of $I_3$ being sparse graphs with  $n$ vertices, of which $n_1$ are of type $1$ and $n_2$ of type $2$). In order to make the paper more self-contained, we opted, however, for a direct proof here.
\end{remark}
Define a decomposition of the edge set of a graph to be a collection
of rigid components such that every edge is exactly in one rigid
component, and such that isolated vertices form their own rigid
components.
\begin{lemma}\label{lem:unique}
  Any graph $G$ decomposes uniquely into rigid components. Any two
  rigid components intersect in at most one vertex.
\end{lemma}
\begin{proof}
  First assume $G$ is sparse.  Consider an edge $e=uv$. The edge $uv$
  itself is a rigid block. The union of all rigid blocks containing
  both $u$ and $v$ is, by Lemma~\ref{lem:intuni} part (i), still
  rigid, and this is the unique maximal block $e$ belongs to. If
  initially we had chosen another edge inside this unique maximal
  block, the result would clearly be the same (if it were larger, we
  could again apply Lemma~\ref{lem:intuni} part (i) and obtain a
  bigger block containing $e$). Thus, the set of edges forms an
  equivalence relation whose equivalence classes are given by the
  rigid components the edges belong to. Isolated vertices always
  belong to their own component, and hence the decomposition is
  unique, proving the first part for sparse graphs. For such graphs,
  the second part of the lemma follows immediately from
  Lemma~\ref{lem:intuni} part (ii). 

  Now, suppose that $G$ is not sparse and consider its rigid
  components: we will show that we can choose one edge $e=uv$,
  remove it from $G$ to obtain $G'=G\setminus uv$, and that the rigid
  components of $G'$ are the same as those of $G$. Then, repeating the
  procedure until the graph is made sparse, the lemma will be proved.
  
 First, since rigidity is monotone, a rigid component of $G'$ is a
  rigid block in $G$, and we have to show only that a rigid component
  in $G$ remains a rigid component in $G'$. Since $G$ is not sparse,
  there exists one subgraph $H$ with $n_1(H)+n_2(H)$ vertices of type
  1 (type 2, resp.) having more than
  $n_1(H)+2n_2(H)-\min\{0,n_1(H)-3\}$ edges. Among all such subgraphs
  choose a minimal subgraph $H$, i.e., any induced subgraph of $H$
  when leaving out at least one vertex is sparse. Note that $H$ is
  rigid.  Choose the edge $e=uv \in E(H)$. $H \setminus uv$ remains
  rigid.  Let $C$ be a rigid component in $G$. We have to show it is a
  rigid component
  in $G'=G\setminus uv$. We consider now three cases:\\
  i) $C$ does not contain both $u$ and $v$: then it remains
  rigid after the removal of $uv$, and there is nothing else to prove.\\
  ii) $C$ contains $u$ and $v$, but not all of $H$. Then in $G'$, $\tilde{H}:=C
  \cap H$ is sparse (as every proper subgraph of $H$, and therefore of
  $H \setminus uv$, is sparse).  By the augmentative property of
  matroids, if not yet spanning, $\tilde{H}$ can be completed to
  obtain a minimally rigid spanning subgraph $C'$ of $C$. Thus $C$
  remains rigid in $G'$.\\
  iii) $C$ contains all of $H$. Since $H \setminus uv$ is rigid, we
  can find a minimally rigid spanning subgraph $\tilde{H}$ of $H
  \setminus uv$. Again, by the augmentative property of matroids, if
  not yet spanning, $\tilde{H}$ can be completed to obtain a minimally
  spanning subgraph $C'$ of $C$, and $C$ remains rigid in $G'$.
\end{proof}

\begin{lemma}\label{lem:addedges}
Take two rigid components $R_1$ and $R_2$. Adding at most
three pairwise disjoint edges $u_iv_i$, with $u_i \in V(R_1)\setminus V(R_2)$, and 
$v_i \in V(R_2)\setminus V(R_1)$ turns $R_1 \cup R_2$ into a rigid block.
\end{lemma}
\begin{proof}
  From the proof of the previous lemma, it suffices to prove the
  statement for sparse $G$. Remember that two rigid components $R_1$
  and $R_2$ intersect in either 0 or 1 vertex, and by monotonicity we
  may assume that there is no edge from $V(R_2) \setminus V(R_1)$ to $V(R_1)
  \setminus V(R_2)$. Let $n_1(R_1)=i$ and $n_1(R_2)=j$. If $R_1 \cap
  R_2=\emptyset$, then let $t=0$, and otherwise let $t \in \{1,2\}$ be
  equal to the type of the vertex in $R_1 \cap R_2$. If $t \in
  \{1,2\}$ and $\min\{i,j\} \geq 3$, by Lemma~\ref{lem:intuni}
  part(ii), $R_1 \cup R_2$ is already rigid. By a similar argument as
  in the proof of Lemma~\ref{lem:intuni} part(ii), we can show that
  the case $i \geq 3, j=2$ and $t=2$ is impossible, as in this case
  $m(R_1 \cap R_2) \geq 1$, and thus $n(R_1 \cap R_2) \geq 2$. In all
  other cases, do the following: if $\min\{i,j\} \geq 3$ and $t=0$,
  then no edge is added. If $i\geq 3$ and $j< 3$, $3-j-t$ edges are
  added, if $i < 3$ and $j < 3$ and $(i+j) \geq 3$, then $6-i-j-t$
  edges are added, and if $i < 3$ and $j < 3$ and $(i+j) < 3$, then
  $3-t$ edges are added.  It can be seen that in all cases the total
  number of edges needed for $R_1 \cup R_2$ being minimally rigid is
  correct. Moreover, the number of edges added is for any fixed value
  of $i$ monotone nondecreasing in $j$. Also, for non-intersecting
  subgraphs $A \subseteq R_1$ and $B \subseteq R_2$ the number of
  edges that can be added between $A$ and $B$ without violating
  sparsity is at least the number that can be added in case they
  intersect in one vertex. In particular, this means that for any
  subgraph $A \subseteq R_1$ with $i' \leq i$ vertices of type $1$ and
  any subgraph $B \subseteq R_2$ with $j' \leq j$ vertices of type $1$
  such that $\min\{n(A),n(B)\} \geq 2$, the number of edges that can
  be added between vertices of $A$ and $B$ without violating sparsity
  is at least the number of edges added between $R_1$ and $R_2$, and
  thus such subgraphs remain sparse. Otherwise, suppose $n(A)=1$
  and we may assume $A$ and $B$ disjoint. By disjointness of the newly
  added edges at most $1$ edge is added between $A$ and $B$. Thus, an
  originally sparse graph $B$ remains sparse after adding one vertex
  and at most one edge. Thus, all subgraphs are sparse, and the
  statement follows.
 \end{proof}

\section{Proof of Theorem~\ref{prop:orientabilite}}\label{sec:orient}

The proof of Theorem~\ref{prop:orientabilite} relies on Theorem 3 in
\cite{Lelarge} and Lemma \ref{lemma_important}.

To a simple graph $G=(V,E)$, we
associate the bipartite graph $G^b=(V^b,E^b)$ with vertex set $V^b=V\cup E$ and
an edge between $v\in V$ and $e\in E$ if and only if $v$ is
an end-point of $e$ in $G$. We say that $G^B$ is the bipartite version
of $G$. The size of a spanning subgraph $S=(V\cup
E, F)$ of $G^b$ is defined as the number of edges $|F|$ of $S$. We now
consider the case where each vertex of the original graph $V$ has a type
in $\{1,2\}$.
We say that a spanning subgraph is admissible if for each $v\in V$, the
degree of $v$ in $S$ is at most its type and for each $e\in E$, the
degree of $e$ in $S$ is at most one. 

Clearly, if $G$ is 1.5-orientable, an orientation gives a spanning
subgraph with size $|E|$ which is the maximum possible size of an
admissible spanning subgraph. Hence we have the following claim: a
graph $G$ is 1.5-orientable if and only if the size of a maximum
admissible spanning subgraph of $G^b$ is equal to $|E|$.

For the random graph $G(n,c/n)$, when types are drawn independently
at random being $1$ with probability $1-q$, and $2$ otherwise, independently
of the rest, we denote by $M_n=M_n(c,q)$ the size of the largest
admissible spanning subgraph of the bipartite version $G^b_n$ of
$G(n,c/n)$. Our previous claim translates into:
$G(n,c/n)=(V_n,E_n)$ is 1.5-orientable if and only if $M_n = |E_n|$.

The fact that $G^b_n$ satisfies the assumption of Theorem 3
\cite{Lelarge} is proved in Section~6 of \cite{Lelarge}. With the
notation of Theorem 3 in \cite{Lelarge}, we choose the set $A_n$ equal
to $E_n$ and the set $B_n$ equal to $V_n$, so that we have $D^A=2$,
$W^A=1$ and
\begin{eqnarray*}
D^B=Poi(c), &\mbox{ and }& W^B = \left\{
\begin{array}{ll}
1& \mbox{, w.p. } 1-q\\
2& \mbox{, w.p. } q.
\end{array}
\right.
\end{eqnarray*}
We now compute $\inf_{x\in [0,1]}\mathcal{F}^A(x)$, which is according to
Theorem 3 in~\cite{Lelarge} the value for the limit $\lim_{n\to
  \infty} \frac{M_n}{|E_n|}$.
Using the definitions given in Theorem 3 of~\cite{Lelarge}, we have
\begin{eqnarray*}
  g^A(x) &=& 1-x\\
  g^B(x) &=& 1-(1-q) Q(cx,1)-qQ(cx,2)\\
  \mathcal{F}^A(x) &=& 1-(1-g^B(x))^2+\frac{2}{c}\left((1-q) Q(cx,2)+2qQ(cx,3)\right).
\end{eqnarray*}
From Theorem 3 in~\cite{Lelarge}, we know that $\mathcal{F}^A(x)$ is
minimized only for $x$ solving $x=g^A\circ g^B(x)$. More precisely,
the derivative of $\mathcal{F}^A(x)$ has the same sign as $\Delta(x)
=x-g^A\circ g^B(x)$.
Then, we have (note that $\frac{d}{dx}Q(x,y) = e^{-x}\frac{x^{y-1}}{(y-1)!}$)
\begin{eqnarray}
\label{eq:defDelta}\Delta(x) &=& x-(1-q) Q(cx,1)-q Q(cx,2)\\
\nonumber\Delta'(x) &=& 1-c(1-q) e^{-cx}-c^2xqe^{-cx}\\
\nonumber\Delta''(x) &=&c^3qe^{-cx}\left(x-\frac{2q-1}{cq}\right)
\end{eqnarray}
Note that $\Delta(0)=0$, $\Delta(1)>0$. Define
$\tilde{x}=\tilde{x}(c,q)$ as the largest solution in $[0,1]$ to the
equation $\Delta(x)=0$.  Note that $\mathcal{F}^A(0) = 1$. Moreover we
have
\begin{eqnarray*}
\Delta'(1) = 1-c(1-q)e^{-c}-c^2qe^{-c}\geq 1-ce^{-c}\geq 1-e^{-1}\geq 0.
\end{eqnarray*}
We now prove that
\begin{eqnarray}
  \label{eq:infF}\inf_{x\in [0,1]}\mathcal{F}^A(x) = \min \{1,\mathcal{F}^A(\tilde{x})\}.
\end{eqnarray}

First assume that $q\leq 1/2$ so that, we have $\Delta''(x)\geq 0$ for
$x\in [0,1]$. If $1-q\leq 1/c$, then $\Delta'(0)=1-c(1-q)\geq 0$ and
we have $\Delta(x)\geq 0$ so that $\tilde{x}=0$ and the function
$\mathcal{F}^A(x)$ is increasing. Hence $\inf_{x\in
  [0,1]}\mathcal{F}^A(x)=\mathcal{F}^A(0) = 1$. Now if $1-q>1/c$, then
a similar analysis shows that $\mathcal{F}^A(x)$ is decreasing on
$[0,\tilde{x}]$ and increasing on $[\tilde{x},1]$, so that we have
$\inf_{x\in [0,1]}\mathcal{F}^A(x)=\mathcal{F}^A(\tilde{x}) < 1$.

Assume now that $q>1/2$ so that $\Delta''(x)$ vanishes once on
$(0,1)$. Hence, if $1-q>1/c$, we have $\Delta'(0)<0$ so that
$\Delta(x)< 0$ for $x\in(0,\tilde{x})$ and $\Delta(x)>0$ for $x\in
(\tilde{x},1]$.  As above, we have $\inf_{x\in
  [0,1]}\mathcal{F}^A(x)=\mathcal{F}^A(\tilde{x}) < 1$.  Consider then
the case $1-q\leq 1/c$, so that $\Delta'(0)\geq 0$. Moreover as
$\Delta'(1)\geq 0$, either $\Delta(x)$ is non-negative or there exists
$0<y<\tilde{x}$ such that $\Delta(x)$ is positive on $(0,y)$ and
$(\tilde{x},1)$ and negative on $(y,\tilde{x})$. In any case, we have
$\inf_{x\in [0,1]}\mathcal{F}^A(x) = \min
\{\mathcal{F}^A(0),\mathcal{F}^A(\tilde{x})\}$, and (\ref{eq:infF}) is
proved.

By Theorem 3 in~\cite{Lelarge}, we have
\begin{eqnarray*}
\lim_{n\to \infty} \frac{M_n}{|E_n|} = \inf_{x\in [0,1]}\mathcal{F}^A(x)=\min
\{1,\mathcal{F}^A(\tilde{x})\}.
\end{eqnarray*}
In the argument above, we showed that for $1-q>1/c$, we have
$\tilde{x}>0$ and $\mathcal{F}^A(\tilde{x})<1$ and for $1/2\leq
1-q\leq 1/c$, we have $\tilde{x}=0$ and $\inf_{x\in
  [0,1]}\mathcal{F}^A(x)=1$.  In particular, we see that for
$1-q>1/c$, $G(n,c/n)$ is not 1.5-orientable and for $1/2\leq 1- q\leq
1/c$, $G(n,c/n)$ is 'almost' 1.5-orientable, i.e., all vertices except
possibly $o(n)$ will satisfy their indegree constraints. We will show
that the graph is indeed 1.5-orientable in the last part of the proof
relying on Lemma~\ref{lemma_important}.

Before that, we consider the case $q>1/2$ and find the condition on
$c$ for $\mathcal{F}^A(\tilde{x})< 1$, where $\tilde{x}=\tilde{x}(c)$
is the largest solution to $\Delta(x)=0$ with $\Delta(x)$ defined in
(\ref{eq:defDelta}).  First note that by the previous analysis, if
$\mathcal{F}^A(\tilde{x})< 1$ and $\Delta(\tilde{x})=0$, then
$\tilde{x}$ is necessarily the largest solution to $\Delta(x)=0$.

Now using the fact that $\Delta(\tilde{x})=0$, we see that
$g^B(\tilde{x})=1-\tilde{x}$, so that we have
\begin{eqnarray*}
  \mathcal{F}^A(\tilde{x}) = 1-\tilde{x}^2 +\frac{2}{c}\left((1-q)Q(c\tilde{x},2)+2qQ(c\tilde{x},3) \right).
\end{eqnarray*}
Making the change of variable $\xi=c\tilde{x}$, we get
\begin{eqnarray*}
  \mathcal{F}^A(\tilde{x}) = 1-\frac{\xi^2}{c^2} +\frac{2}{c}\left((1-q)Q(\xi,2)+2qQ(\xi,3) \right).
\end{eqnarray*}
Define the function 
\begin{eqnarray*}
f(\xi,q) = \xi\frac{(1-q)Q(\xi,1)+qQ(\xi,2)}{(1-q)Q(\xi,2)+2qQ(\xi,3)}
\end{eqnarray*}
and recall that the definition of $\tilde{x}$ becomes for $\xi$:
\begin{eqnarray}
\label{eq:defc}\xi =c\left((1-q)Q(\xi,1)+qQ(\xi,2) \right),
\end{eqnarray}
so that we can rewrite the previous expression as
\begin{eqnarray*}
\mathcal{F}^A(\tilde{x}) = 1-\frac{1}{c} \left(
  f(\xi,q)-2\right)\left((1-q)Q(\xi,2)+2qQ(\xi,3)\right).
\end{eqnarray*}
We have $\mathcal{F}^A(\tilde{x}) < 1$ if and only if
$f(c\tilde{x},q)>2$.  Note that for $q>0$, we have $\lim_{\xi\to
  0}f(\xi,q)=2$ and $\lim_{\xi\to 0}f(\xi,0)=3/2$. Moreover, for any
$q>1/2$, one can show that there exists a unique positive solution to
$f(\xi^*,q)=2$ and for $\xi\in (0,\xi^*)$, we have $f(\xi,q)<2$, while
for $\xi>\xi^*$, we have $f(\xi,q)>2$.  Now, using (\ref{eq:defc}), we
define
\begin{eqnarray*}
c^* = \frac{\xi^*}{(1-q)Q(\xi^*,1)+qQ(\xi^*,2)},
\end{eqnarray*}
so that $c^*\tilde{x}(c^*)=\xi^*$.

Note that as a function of $c$ with $x$ fixed, $\Delta(x)$ is
non-increasing in $c$ which implies that $\tilde{x}(c)$ is
non-decreasing in $c$ (note that for $q>1/2$, by the previous
analysis, $\Delta(x)$ is always negative on the left of $\tilde{x}$
and positive on its right).

Thanks to the monotonicity of $c\mapsto \tilde{x}(c)$, we have: if
$c>c^*$, then $c\tilde{x}(c) >\xi^*$ and $f(c\tilde{x},q)>2$, i.e.,
$\mathcal{F}^A(\tilde{x}) < 1$.  Similarly, if $c<c^*$ and
$\tilde{x}(c)>0$, then we get $f(c\tilde{x},q)<2$, i.e.,
$\mathcal{F}^A(\tilde{x}) > 1$.

To summarize, we proved that for the function $c^*(q)$ defined in the
statement of the theorem, we have: for $c<c^*(q)$, a.a.s., we have $\lim_{n\to
  \infty} \frac{M_n}{|E_n|} =1$ and for $c>c^*(q)$, we have
$\lim_{n\to \infty} \frac{M_n}{|E_n|} <1$.  In particular, (b)
follows: if $c>c^\ast(q)$, the graph is a.a.s. not 1.5-orientable. We
still need to prove that if $c<c^\ast(q)$, the graph is a.a.s.
1.5-orientable.

Choose $\tilde{c}$, such that $c<\tilde{c}<c^\ast$. Let $\tilde{G}_n$
be an associated random graph and $\tilde{M}_n$ a maximum admissible
subgraph.  Construct a coupling between random graphs with different
parameters $c<\tilde{c}$, by removing edges in $\tilde{G}_n$ with the
appropriate probability.  The goal is to construct an admissible
subgraph $\bar{M}_n$ for some $G_n(\bar{c})$ with $\bar{c}\geq c$ such
that $|\bar{M}_n|=|\bar{A}_n|$. In other words, $G_n(\bar{c})$ is
1.5-orientable which implies the claim as $G_n(n,c/n)$ can be obtained
from $G_n(\bar{c})$ by removing edges.

If $|\tilde{M}_n|=|\tilde{A}_n|$, we are done. Assume then that
$|\tilde{M}_n|<|\tilde{A}_n|$ and consider the bipartite graph
$\tilde{G}_n^b$. We say that a vertex $v\in \tilde{V}_n$ is saturated
if its degree in $\tilde{M}_n$ is equal to its type. Note that if an
edge $(v,e)\in E(\tilde{G}_n^b)$ where $v\in \tilde{V}_n$ and $e\in
\tilde{E}_n$ is not covered by $\tilde{M}_n$, then the vertex $v$ is
saturated (otherwise $\tilde{M}_n$ would not be maximal hence not
maximum).  In particular, if $e\in \tilde{E}_n$ is isolated in
$\tilde{M}_n$, then each of its neighbors $u$ and $v$ is
saturated. Starting from these vertices and following the covered
edges, we can then construct alternating paths in which the edges are
alternatively covered in $\tilde{M}_n$ and uncovered. Let
$\tilde{K}^b$ be the union of all such alternating paths (see
Figure~\ref{fig:alt} for an illustration).
\begin{figure}
\includegraphics[width=10cm]{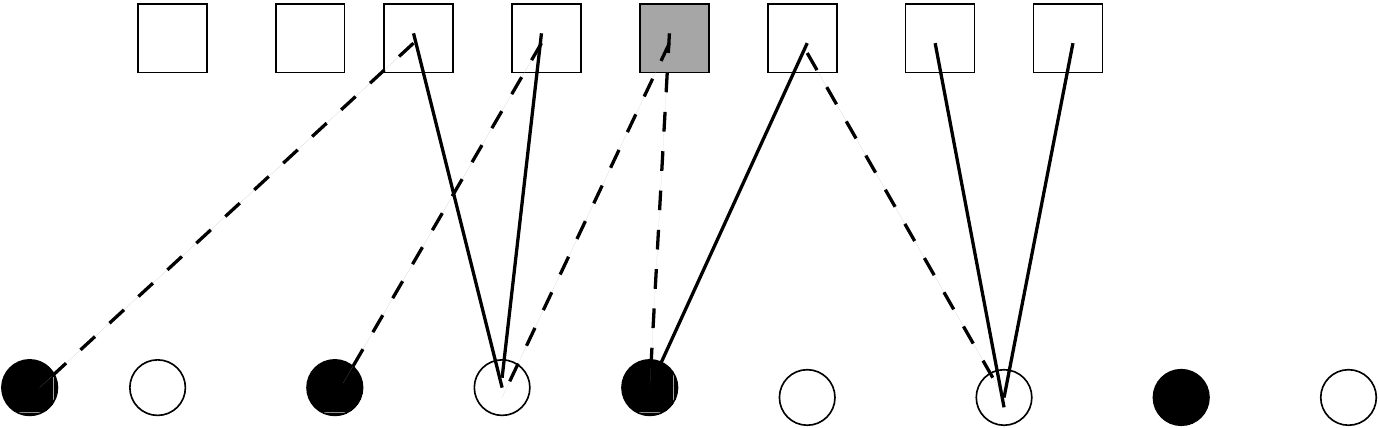}
\caption{Bipartite representation of $G_n$ and construction of the
  alternating paths.  Squares (upper row) are the edges of $G_n$;
  circles (lower row) are the vertices of $G_n$. Filled (resp. empty)
  circles are vertices of type 1 (resp. type 2). Solid lines (resp. dashed lines)
  are edges in $\tilde{G}_n^b$ that are (resp. are not) in
  $\tilde{M}_n$. Note that not all edges of $\tilde{G}_n^b$ are
  represented. The grey square is an edge of $G_n$ which is not
  covered by $\tilde{M}_n$: thus it is a starting point for
  alternating paths. \label{fig:alt}}
\end{figure}

Each vertex $v\in \tilde{V}_n\cap \tilde{K}^b$ is saturated so that
the graph $\tilde{K}$ associated to $\tilde{K}^b$ in the original
$\tilde{G}_n$ is a subgraph satisfying the hypotheses of
Lemma~\ref{lemma_important}. Hence there exists $\alpha>0$ such that
$\tilde{K}$ has size (i.e., number of vertices $|V(\tilde{K})|$) at
least $\alpha n$.

Let us call ${\mbox gap}_n=|\tilde{E}_n|-|\tilde{M}_n|$. By the
previous analysis, we know that ${\mbox gap}_n=o(n)$.  To make the
coupling between random graphs at different $c$'s explicit, attach a
uniform random variable $\mathcal{U}_{[0,1]}$ to each edge of
$\tilde{G}_n$, rank the edges according to these variables, and delete
them sequentially to construct the graphs $G_n(c)$ with
$c<\tilde{c}$. Since $|V(\tilde{K})|\geq \alpha n$, each time an edge
is removed, the probability it belongs to $\tilde{K}$ is larger than
$\varepsilon' $, for some $\varepsilon'>0$.  Moreover, each time an
edge in $\tilde{K}$ is removed, ${\mbox gap}_n$ decreases by one. Thus
the probability that deleting an edge decreases ${\mbox gap}_n$ is at
least $\varepsilon' $. Hence a.a.s.  ${\mbox gap}_n$ reaches $0$
before the graph $G_n$ with parameter $c$ is constructed. At this
point, we have found $\bar{c}\geq c$ such that
$|\bar{M}_n|=|\bar{A}_n|$ and we have proved that a.a.s. $G_n(n,c/n)$ is
$1.5$-orientable: we are done.

\section{Orientability and rigidity}\label{sec:or-rig}
In this section we will show that the threshold of having a giant
rigid component coincides with the threshold for $1.5$-orientability,
which together with Theorem~\ref{prop:orientabilite} completes the
proof of Theorem~\ref{theo_threshold}.

One part of the equivalence of the thresholds of rigidity
and $1.5$-orientability is easy:
\begin{lemma}\label{lem:rigid_nonorientable}
  If for some $c$, $G$ contains a rigid component $H$ of linear size,
  then a.a.s., for all $c' \geq c+\epsilon$ and
  $\epsilon > 0$, $G$ is not $1.5$-orientable.
\end{lemma}
\begin{proof}
  This lemma is similar to Lemma~5.1 of~\cite{Moore}. If $G$ contains
  a giant rigid component $H$, then $|V(H)|=n'_1+2n'_2 \geq \alpha n$
  for some $\alpha > 0$ and $|E(H)| \geq n'_1+2n'_2-3$. When
  considering $c'=c+\epsilon$ for some $\epsilon > 0$, then note that
  a graph in $G' \in \mathcal{G}(n,c'/n)$ can be obtained from $G$ by
  adding a fresh random graph with parameters $G'' \in
  \mathcal{G}(n,\epsilon(1+o(1))/n)$, where the $o(1)$ part accounts
  for edges present in the intersection of both graphs. When adding
  the fresh random graph, $\Theta(n)$ edges will be added with
  probability $1-e^{-\Omega(n)}$. For each such edge, there is
  positive probability that both of its endpoints are in $H$, and thus
  for any fixed $\epsilon > 0$, a.a.s., at least $4$
  edges will be added to $E(H)$, and for $c'$,
  $|E(H)|>n'_1+2n'_2$. Using Proposition~\ref{prop:flow}, $G$ is not
  1.5-orientable, and the lemma follows.
\end{proof}
\medskip The other direction is harder. For $q=1$ (all sites are type
2, standard 2D rigidity percolation), the authors of~\cite{Moore} use
a lemma by Theran (\cite{Theran}), stating that rigid components
have size at most $3$, or they are of size $\Omega(n)$. This is not
true for $q<1$, and we will use Lemma~\ref{lemma_important} instead.
We will first make one simple observation.

\begin{lemma}\label{lem:rigidgiant}
  Let $q < 1$ and let $G \in \mathcal{G}(n,c/n)$ with $c >
  \frac{1}{1-q}$. A.a.s., $G$ contains a giant rigid connected block.
\end{lemma}
\begin{proof}
  We will show that the subgraph induced by the vertices of type $1$
  contains a giant rigid block: indeed, note that the subgraph
  $G' \in \mathcal{G}((1-q)n, c/n)= \mathcal{G}((1-q)n,
  \frac{c(1-q)}{(1-q)n})$. Since we are interested in the asymptotic
  behavior of such graphs, we may replace $(1-q)n$ by $n$, and the
  behavior of such a graph is like the one of $G' \in \mathcal{G}(n,
  \frac{c(1-q)}{n})=\mathcal{G}(n,c'/n)$ for some $c' > 1$. By
  standard results (see for example~\cite{Bollobas}), a.a.s.,
  $\mathcal{G}(n,\frac{c' - (c'-1)/2}{n})$ contains a giant
  connected component, and by adding a fresh random graph $G''
  \in \mathcal{G}(n, \frac{(1+o(1))(c'-1)/2}{n})$, a.a.s.  at least
  one cycle will be added. Hence, by possibly removing edges, we may
  pick a connected subgraph $H$ of linear size containing exactly one
  cycle. $H$ is the desired rigid block, since $|V(H)|=|E(H)|$,
  and every subgraph of $H$ satisfies the sparsity condition. The
  statement follows.
\end{proof}
We need one more helper lemma.  
\begin{lemma}\label{prop1}
There is $\alpha>0$ such that:
 if $G$ is a $G(n,c/n)$ random graph with $c>2$, and $G$ is not Laman
  sparse, then, a.a.s., $G$ spans a rigid connected block of
  size at least $\alpha n$.
\end{lemma}
\begin{proof}
  Assume first that $G$ is a $G(n,c_0/n)$ random graph with $c_0>2$,
  and has only type 2 vertices.  If $G$ is not Laman sparse, then by
  Lemma~4.2 of~\cite{Moore}, $G$ spans a rigid component on at least
  four vertices. Now, by Proposition~3.3 of~\cite{Moore}, a.a.s. all
  rigid components of $G$ have size 1,2,3 or $\Omega(n)$. In fact, by
  applying Proposition~4 of~\cite{Theran} with $a=2$, we see that
  a.a.s. such a rigid component is of size at least $((4/c_0)^2
  e^{-3})n$. Since rigidity is preserved by addition of edges,
this statement is
  also true for any $c\geq c_0$. Note that these rigid
    components with only type 2 vertices are connected. Now,
  according to Lemma~\ref{lem:changementtype} a rigid subgraph of $G$
  remains rigid if some vertices of $G$ are changed from type 2 to
  type 1.  Hence, for any $c \geq c_0>2$, $G \in \mathcal{G}(n,c/n)$
  spans a rigid connected block which is of size at least $((4/c_0)^2
  e^{-3}) n$. The statement of the lemma follows with $\alpha
  =4e^{-3}$, taking the limit as $c_0 \to 2$.

\end{proof}

We are now able to prove the counterpart to
Lemma~\ref{lem:rigid_nonorientable}.

\begin{lemma}\label{lem:nonorientable_rigid}
  Suppose that $G \in \mathcal{G}(n,c/n)$ is not
  $1.5$-orientable. Then, for any $\epsilon > 0$, a.a.s., $G \in
  \mathcal{G}(n,(c+\varepsilon)/n)$ contains a giant rigid
  connected block $H$.
\end{lemma}
\begin{proof}
Assume now that $G$ is not
  1.5-orientable. We have shown in Theorem~\ref{prop:orientabilite}
  that the threshold for $1.5$-orientability $c^*(q)$ satisfies
  $c^*(q) = \frac{1}{1-q}$ for $q \in [0,1/2]$ and $c^*(q) <
  \frac{1}{1-q}$ for $q \in (1/2,1]$, and since we are interested in a
  statement that holds a.a.s., we may assume $c \geq
  c^*(q)$. For $q \in [0,1/2]$, for any $\varepsilon > 0$,
  $c+\varepsilon > \frac{1}{1-q}$, and by Lemma~\ref{lem:rigidgiant},
a.a.s. $G$ contains a connected giant rigid component. 

We may therefore assume $q > 1/2$ and $c < \frac{1}{1-q}$ and have to
show that in this case a.a.s. $G \in \mathcal{G}(n,c/n)$ contains a
giant rigid connected block $H$, implying the statement since the property
of containing a giant rigid connected block is monotone. Notice that
for $q>1/2$, $c^*(q)>2$; since we may assume $c\geq c^*(q)$, we may
assume $c>2$.  By Proposition~\ref{prop:flow}, there exists a subgraph
$H\subseteq G$ with $n'_1$ vertices of type $1$, $n'_2$ vertices of
type 2, such that $|E(H)|=m'>n'_1+2n'_2$. Among all such subgraphs,
let $H$ be minimal with respect to the number of vertices for this
property (if there are several choices, pick an arbitrary such
$H$). By minimality, $H$ is connected. Since now $c <
\frac{1}{1-q}$, by Lemma~\ref{lemma_important}, a.a.s., $|V(H)|\geq
\alpha n$ for some $\alpha=\alpha(q,c-\frac{1}{1-q})$. Now we have to
construct a giant rigid connected block starting from $H$. We will now show
that either $G$ contains a connected giant rigid block or any arbitrary
subgraph $\tilde{H} \subseteq H$ with $V(\tilde{H})\neq V(H)$ fulfills
the sparsity condition.  Consider an arbitrary subgraph $\tilde{H}
\subseteq H$ with $V(\tilde{H})\neq V(H)$, and let $\tilde{n}_1$ and
$\tilde{n}_2$ be the numbers of type $1$ and type $2$ vertices of
$\tilde{H}$, respectively. By minimality of $H$,
\[
|E(\tilde{H})|\leq \tilde{n}_1+2\tilde{n}_2.
\]
Consider first the case $\tilde{n}_1=0$. If $|E(\tilde{H})| \geq
2\tilde{n}_2-2$, $G$ is not Laman sparse, $c>2$, hence Lemma~\ref{prop1}
ensures that $G$ contains a giant rigid connected block, and we are
done. Otherwise, $|E(\tilde{H})|\leq 2\tilde{n}_2-3$ for any subgraph
$\tilde{H}$ with $\tilde{n}_1=0$, and these subgraphs are thus
sparse. In this case we consider then subgraphs with
$\tilde{n}_1=1$. If $|E(\tilde{H})|>2\tilde{n}_2-1$, then
$|E(\tilde{H})|>2(\tilde{n}_1+\tilde{n}_2)-3$, and again $G$ is not
Laman sparse, and as before Lemma~\ref{prop1} ensures that $G$
contains a giant rigid connected block, and we are done. Otherwise,
$|E(\tilde{H})|\leq 2\tilde{n}_2-1$ for all subgraphs with
$\tilde{n}_1=1$ holds, and these subgraphs are sparse as well. In this
case we consider the case $\tilde{n}_1=2$. If
$|E(\tilde{H})|>2\tilde{n}_2+1$, then
$|E(\tilde{H})|>2(\tilde{n}_1+\tilde{n}_2)-1$, and by the same
argument as before, by Lemma~\ref{prop1}, $G$ contains a giant rigid connected block, and we are done. Otherwise, for all subgraphs with
$\tilde{n_1}=2$, we have $|E(\tilde{H})|\leq 2\tilde{n}_2+1$, and
these subgraphs are sparse as well. We then consider subgraphs
$\tilde{n}_1\geq 3$. For them, $\tilde{H}$ clearly verifies the
sparsity condition.  Thus, either we have found a giant rigid connected block in $G$, or all proper subgraphs of $H$ fulfill the sparsity
condition. In the latter case, since $|E(H)|>n'_1+2n'_2$, we can
remove some edges from $H$ so that $|E(H)|=n'_1+2n'_2$ (in the case
$n'_1 \geq 3$), or $|E(H)|=n'_1+2n'_2-1$ (in the case $n'_1=2$), or
$|E(H)|=n'_1+2n'_2-2$ (in the case $n'_1=1$), or $|E(H)|=2n'_2-3$ (in
the case $n'_1=0$), and in all cases, since $|E(H)|\geq |V(H)|$, this can be done without disconnecting $H$. In this way $H$ is minimally rigid, connected, and since
$|V(H)|\geq \alpha n$ for some $\alpha=\alpha(q,c-\frac{1}{1-q})$, $H$
provides the giant rigid connected block.
\end{proof}
Combining Lemma~\ref{lem:rigid_nonorientable} and
Lemma~\ref{lem:nonorientable_rigid}, we see that the thresholds for
$1.5$-orientability and rigidity coincide, and together with
Theorem~\ref{prop:orientabilite} the proof of the first part of
Theorem~\ref{theo_threshold} is completed.  

We turn now to the uniqueness statement of Theorem~\ref{theo_threshold},
adapting the proof of \cite{Moore}. Note first that since the size of
each giant rigid connected block is at least $n/\omega_n$ for some
function $\omega_n$ that arbitrarily slowly tends to $\infty$ as $n \to \infty$,
and since by Lemma~\ref{lem:intuni}(i) any two rigid blocks intersect
in at most one vertex, there can be at most $\omega_n(1+o(1))$ such
blocks (we suppose connected blocks here to be giant inclusion-wise maximal connected blocks).

We generate $G \in \mathcal{G}(n,c/n)$ as follows. Start with the
empty graph; order randomly the edges in $K_n$; then add sequentially
the first $m$ edges according to this ordering, with $m \sim
Bin(\binom{n}{2},c/n)$. At time $t$, the graph under construction thus 
has $t$ edges. 


Define $s:=\omega_n^3 \log n$ and suppose now that at some time $t_0$ we have two connected rigid blocks $R_1$ and $R_2$, each  of size at least
$n/\omega_n$ for some function $\omega_n$ tending to infinity arbitrarily slowly as $n \to \infty$, such that $R_1 \cup R_2$ is not yet a connected rigid block. 
Lemma~\ref{lem:addedges} ensures that is enough to add $3$ pairwise
disjoint edges between  $R_1$ and
$R_2$ to make $R_1 \cup R_2$ a giant connected rigid block. The probability that $R_1 \cup R_2$ is not a connected rigid block by time $t_0+s$ is at most
$$
(1- \Omega((\frac{1}{\omega_n}))^2)^{s}= O(n^{-\omega_n}).
$$
Since there are at most $r:=\omega_n(1+o(1))$ (giant inclusion-wise maximal) connected rigid blocks, after $r-1$ mergings the union of all such blocks forms a unique
connected rigid block: this can be seen by considering an auxiliary graph whose vertices are giant connected rigid blocks, and an edge between two vertices is added if the union of the blocks is a connected rigid block; once this auxiliary graph is a tree, the union of all blocks is also a connected rigid block. Therefore, with probability $1-O(\omega_n n^{-\omega_n})=1-O(n^{-\omega_n})$ there are in total at most
$\omega_n(1+o(1)) s$ steps with more than one giant connected rigid block.

 The probability that $m$ equals any
fixed number of edges is $O(1/\sqrt{n})$, and hence the probability
that $m$ equals one fixed number having one more than one giant
connected rigid block is at most $O(s \omega_n/\sqrt{n})=o(1)$. Since
after exactly $m$ steps we obtain $G \in
\mathcal{G}(n,c/n)$, a.a.s. there is only one unique
rigid giant connected block in $G$, and thus a.a.s. clearly also one
unique giant rigid component, and Theorem~\ref{theo_threshold}
follows.

\section{Proof of Theorem~\ref{theo_continuous}}\label{sec:theo_cont}


We first prove the following auxiliary lemmas. 
\begin{lemma}\label{lemma_degree}
  Let $G$ be a rigid graph. Then for $i \in \{1,2\}$, any vertex of
  type $i$ has degree at least $i$ in $G$.
\end{lemma}
\begin{proof}
  Let $\bar{G}$ be a spanning minimally rigid subgraph of $G$, and
  suppose that $v$ is a vertex of type $i$ with degree strictly less
  than $i$ in $\bar{G}$.  Considering the cases $n'_1=0,1,2$ and $n'_1
  \geq 3$ separately, we see that in each case $\bar{G} \backslash
  \left\{v \right\}$ is a subgraph of $\bar{G}$ which violates the
  sparsity condition, contradicting the sparsity of $\bar{G}$.
\end{proof}
\begin{lemma}\label{lemma_removal}
Let $G$ be a rigid graph, and $v$ a type $i$ vertex with degree $i$.
Then $G\backslash \{v\}$ is rigid (but not necessarily connected).
\end{lemma}
\begin{proof}
Let $\bar{G}$ be a spanning, minimally rigid subgraph of $G$. The degree of 
$v$ in $\bar{G}$ is by Lemma~\ref{lemma_degree} still $i$. 

First, $\bar{G}\backslash \{v\}$ is sparse, since $\bar{G}$ is sparse.
If $v$ is of type $2$, removing $v$ removes two degrees of freedom and
two edges, so the constraint of the total number of edges counting for
minimal rigidity is satisfied for $\bar{G} \setminus \{v\}$. Note, however, that
$\bar{G}\backslash \{v\}$ might be the disjoint union of two rigid
blocks with no edge in between (in which case, by sparsity, both
blocks $H_1$ and $H_2$ satisfy $n_1(H_i)\geq 3$). If $v$
is type $1$ and if also $n_1(G)>3$, then removing
$v$ removes one degree of freedom and one edge, so the constraint
counting for minimal rigidity is also satisfied for $\bar{G} \setminus
\{v\}$.  

It remains to consider the
case $v$ of type $1$ and $n_1(G)\leq 3$.  If we had $n_1(G)=1$, then
$\bar{G}\backslash \{v\}$ is sparse with $n_1(\bar{G}\backslash
\{v\})=0$, and we would have
\[
m(\bar{G}\backslash \{v\})\leq 2n(\bar{G}\backslash \{v\})-3.
\]
 This would imply for $\bar{G}$
\[
m(\bar{G})=m(\bar{G}\backslash \{v\})+1 \leq 2n_2(\bar{G}) -2 
=n_1(\bar{G})+2n_2(\bar{G})-3,
\]
contradicting the rigidity of $\bar{G}$. The same reasoning excludes 
the cases $n_1(G)=2$ and $n_1(G)=3$, and the proof follows.
\end{proof}

\begin{lemma}\label{lemma_core2}
  Let $C$ be a rigid block of a graph $G$, $K_{2.5}$ the 2.5-core of $G$, and
  $K^+_{2.5}$ its $2.5+1.5$-core. If the $2.5$-core of $C$ is not
  empty, then $C \subseteq K^+_{2.5}$.
\end{lemma}
\begin{proof}

  Let us recursively remove vertices to construct the 2.5-core of $C$.
  $C_0:=C$ is rigid, thus by Lemma~\ref{lemma_degree}, the first
  removed vertex $v_1$, assumed to be of type $i$, has degree $i$. By
  Lemma~\ref{lemma_removal}, $C_1=C\backslash \left\{ v_1\right\}$ is
  rigid. We can iterate this process until the $2.5$-core of $C$ is
  constructed. At each step we can apply Lemmas~\ref{lemma_degree} and
  \ref{lemma_removal}, hence we only remove type $i$ vertices with
  degree exactly $i$. The union of all remaining vertices remaining
  form the $2.5$-core of $C$. Since the $2.5$-core of $C$ is not empty
  and is a subgraph of $K_{2.5}$, all the removed vertices are in the
  $2.5+1.5$-core. Therefore $C \subseteq K^+_{2.5}$.

\end{proof}
Collecting all previous results, it is now not hard to prove
Theorem~\ref{theo_continuous}.

Consider first the case $q > 1/2$. Since the threshold for the
existence of a giant rigid component $c^*(q)$ satisfies $c^*(q) <
\frac{1}{1-q}$, by monotonicity of the property of having a giant
rigid component, it suffices to show the claim for $c <
\frac{1}{1-q}$. By looking at the proof of
Lemma~\ref{lem:nonorientable_rigid}, either $G$ is not Laman-sparse,
and then by Lemma~\ref{lem:rigidgiant}, there exists some $\alpha > 0$
such that the size of the giant rigid component is at least $\alpha
n$, or a rigid giant component $H$ is found, which by
Lemma~\ref{lemma_important} is also of size at least $\alpha n$ for
some $\alpha > 0$. Thus, the transition is discontinuous for all $q >
1/2$, and the first part of the statement of
Theorem~\ref{theo_continuous} follows.

Consider now the case $q \leq 1/2$. Denote by $H$ the largest rigid
component. Consider $c=\frac{1}{1-q} +\varepsilon$ for any fixed
$\varepsilon > 0$.  Take $c=\frac{1}{1-q} +\varepsilon$,
$\varepsilon>0$. By Lemma~\ref{lem:rigidgiant}, for $c=\frac{1}{1-q}
+\varepsilon/2$, the subgraph of the vertices of type $1$ contains a
giant component. Moreover, by a result of~\cite{AjtaiKomlosSzemeredi},
this subgraph contains a.a.s. a path of length at least $(\varepsilon/2)^2
(1-q)n/5$. By adding a fresh random graph $G'' \in \mathcal{G}(n,
\frac{(1+o(1))(\varepsilon/2)}{n})$ (as before, the $o(1)$ term is for
the intersection, so that we end up with $G \in \mathcal{G}(n,c/n)$),
a.a.s. among these vertices a cycle of length $\mu
\varepsilon^2 n$ for some small $\mu > 0$ is created. Since this is a
cycle containing vertices of type $1$ only, this cycle is a rigid block, and
hence $R_n(q,c)/n$ is at least $\mu \varepsilon^2$. Imagine now that
this cycle of type 1 vertices is not included in $H$: 
if one of the vertices of the cycle would be included in $H$ and $n_1(H) \geq 3$, then we could add all
other vertices of the cycle in a path like way, that is, for each vertex of the
cycle not yet present add the vertex together with exactly one
incident edge, and the graph remains minimally rigid (in particular,
if for a vertex both of its neighbors on the cycle are already there,
add just one edge; in particular, the cycle of type $1$ is contained in $H$). Similarly, if one vertex of the cycle would be included in $H$ and $n_1(H)=2$, we could add all vertices of the cycle, including the edge closing the cycle (so that the total number of edges is right). If one vertex of the cycle would be included in $H$ and $n_1(H)=1$, by adding
$\sqrt{n}$ random edges, one would modify $c$ by $O(1/\sqrt{n})$ only,
and a.a.s. induce at least one edge between a vertex of the cycle
of type 1 vertices not yet in $H$ and a vertex of type $2$ already in $H$, and sparsity and minimal rigidity remain true.
So suppose no vertex is included in $H$, which is
then also at least of size $\mu \varepsilon^2 n$. Then by adding
$\sqrt{n}$ random edges, one would modify $c$ by $O(1/\sqrt{n})$ only,
and a.a.s. induce at least three pairwise vertex-disjoint edges between the cycle
of type 1 vertices and $H$. If we had $n_1(H) \geq 3$, keep one of the newly added
edges and remove one edge of the cycle, if $n_1(H)=2$, keep one added edge, if $n_1(H)=1$, keep two edges, and if $n_1(H)=0$, keep three edges, without removing any edge in the last three cases. One can check that in all cases one obtains a minimally rigid block, as the total number of edges is correct and every subgraph is sparse. 
 Hence we may assume that the cycle of type 1 vertices is included in $H$. Since
the cycle of vertices of type $1$ forms part of the $2.5$-core, we
know that the $2.5$-core of $H$ is not empty.  Then by
Lemma~\ref{lemma_core2}, $H \subseteq K_{2.5}^+$. Now, by
Theorem~\ref{theo_core+} and Remark~\ref{rem_core+}, we know that the
size of $K_{2.5}^+/n$ tends to 0 when $c\to \frac{1}{1-q}$, and the
second part of the statement of Theorem~\ref{theo_continuous} follows.

\section{Proof of Theorem~\ref{theo_core}}\label{sec:theo_core}
\begin{proof}
The proof is an easy generalization of~\cite{Janson}, see \cite{lelarge_diff}. We repeat the argument here for the convenience
of the reader.

The 2.5-core of an arbitrary finite graph can be found by removing
vertices of type 1 with degree $<2$ and vertices of type 2 with degree
$<3$, in arbitrary order, until no such vertices exist. Let us call a
vertex of type 1 with degree $<2$ or of type 2 with degree $<3$ a
light vertex and let us call it a heavy vertex otherwise. 
We still obtain the 2.5-core by removing edges where one endpoint is
light.

Regard each edge as consisting of two half-edges, each half-edge having
one endpoint. We say that a half-edge is light or heavy when its
endpoint is. As long as there is any light half-edge, choose one such
half-edge uniformly at random and remove the edge it belongs to. When
there are no light half-edges left, we stop. Then all light vertices
are isolated and the heavy vertices with the remaining edges form the
2.5-core of the original graph.

We apply this algorithm to a random multigraph with given degree
sequence $(d_i)_1^n$ (see \cite{Janson}, Section~2 for a precise definition).
We observe the half-edges but not how they are connected into
edges. At each step, we thus select a light half-edge at random. We
then reveal its partner, which is random and uniformly distributed
over the set of all other half-edges. We then remove these two
half-edges and repeat as long as there is any light half-edge.

We now regard vertices as bins and half-edges as balls. Each bin
inherits the type of its vertex.
In each step, we remove first one random ball from the set of balls in
light bins (i.e., bins of type 1 with $<2$ balls or bins of type $2$
with $<3$ balls) and then a random ball without restriction. We stop
when there are no non-empty light bins and the 2.5-core consists
precisely of the heavy bins at the time we stop.

We thus alternately remove a random light ball and a random ball. We
may just as well say that we first remove a random light ball. We then
remove balls in pairs, first a random ball and then a random light
ball, and stop with the random ball leaving no light ball to remove.

We now run this deletion process in continuous time such that, if
there are $j$ balls remaining, then we wait an exponential time with
mean $1/j$ until the next pair of deletions. In other words, we make
deletions at rate $j$. Let $L(t)$, $H(t)$ denote the numbers of
light and heavy balls at time $t$, respectively; further let $H_1(t)$
and $H_2(t)$ be the number of heavy bins of type 1 and 2, respectively.

Let $\tau$ be the stopping time of this process. As in \cite{Janson},
we first consider the total number of balls. This is a death process
with rate 1 and jumps of size 2, so that by Lemma~4.3 in \cite{Janson},
we have:
\begin{eqnarray*}
\sup_{t\leq \tau} |L(t)+H(t) -2me^{-2t}| =o_p(n).
\end{eqnarray*}

We now concentrate on heavy balls. As shown in~\cite{Janson} (see
Section~6), the same results can be applied if the degree sequence is
not given, but converges in probability. In particular, the degree
sequence of $G(n,c/n)$ is random, and it converges in probability to a
Poisson distribution with mean $c$.  Let $U^1_r(t)$ (resp. $U^2_r(t)$)
be the number of heavy bins of type 1 (resp. 2) with exactly $r$ balls
at time t. Then by Lemma~4.4 in~\cite{Janson}, we get (we use
the fact that $\sum_{k\geq j}ke^{-\lambda}\frac{\lambda^k}{k!} =
\lambda Q(\lambda,j-1)$):
\begin{eqnarray*}
\sup_{t\leq \tau} |\sum_{r\geq 2} rU^1_r(t)/n - (1-q)
ce^{-t}Q(ce^{-t},1)|&=&o_p(1)\\
\sup_{t\leq \tau} |\sum_{r\geq 3} rU^2_r(t)/n - q ce^{-t}Q(ce^{-t},2)|&=&o_p(1)
\end{eqnarray*}

We define
\begin{eqnarray*}
h(x) &=& (1-q) cxQ(cx,1) + qcxQ(cx,2)\\
h_1(x) &=& (1-q)Q(cx,2)\\
h_2(x)&=& qQ(cx,3).
\end{eqnarray*}
Since we have $H(t)=\sum_r rU^1_r(t)+rU^2_r(t)$, we get:
\begin{eqnarray*}
\sup_{t\leq \tau}|H(t)/n - h(e^{-t})|  = o_p(1)\\
\sup_{t\leq \tau}|H_1(t)/n - h_1(e^{-t})|  = o_p(1)\\
\sup_{t\leq \tau}|H_2(t)/n - h_2(e^{-t})|  = o_p(1)
\end{eqnarray*}
Hence we deduce that
\begin{eqnarray*}
\sup_{t\leq \tau} |L(t)/n + h(e^{-t}) - ce^{-2t}| = o_p(1).
\end{eqnarray*}
In case (a), we have $cx^2-h(x)>0$ for all $x>0$, so that as in
\cite{Janson}, we conclude that $\tau\to \infty$ a.a.s.,
$H(\tau)/n=H_1(\tau)/n = H_2(\tau)/n=o_p(1)$, and hence case (a)
follows.  In case (b), again following \cite{Janson}, we have
$\tau\to-\log\left(\tilde{\xi}/c \right)$, and the claim follows.
\end{proof}

\section{Proof of Theorem~\ref{theo_core+}}\label{sec:theo_core+}

To prove Theorem \ref{theo_core+}, we need to prove that for a pair of vertices
$a$ and $b$ chosen uniformly at random, we have:
\begin{eqnarray*}
\Pr{(a \mbox{ in 2.5+1.5-core}) } = (1+o(1))\left(1-e^{-\tilde{\xi}}-q \tilde{\xi} e^{-\tilde{\xi}}\right),
\end{eqnarray*}
and 
\begin{eqnarray}
\label{eq:cheb}\Pr{(a \mbox{ and }b \mbox{ in 2.5+1.5-core}) } \leq (1+o(1))\left(1-e^{-\tilde{\xi}}-q \tilde{\xi} e^{-\tilde{\xi}}\right)^2,
\end{eqnarray}
and the statement follows by Chebyshev's inequality. 

To prove the first statement, we first consider the extended $2.5$-core
$C(a)$ obtained as in the previous section by removing
vertices of type 1 with degree $<2$ and vertices of type 2 with degree
$<3$ except for node $a$, that is, $a$ is considered as always heavy. 
Clearly the resulting graph contains the $2.5$-core.
Note that if we condition the resulting
graph on its degree sequence, it is still a configuration model (see
Theorems 10 and 11 in \cite{lelarge_diff}).

We have:
\begin{itemize}
\item[(i)] if $a$ is of type 1 (resp. 2) and has degree $0$ (resp. $0$
  or $1$) in $C(a)$, then $a$ is not in the $2.5+1.5$-core.
\item[(ii)] if $a$ is of type 1 (resp. 2) and has degree $\geq 2$
  (resp. $\geq 3$) in $C(a)$, then $a$ is in the $2.5$-core.
\item[(iii)] if $a$ is of type 1 (resp. 2) and has degree $1$ (resp. $2$)
  in $C(a)$, then we can remove $a$ and then continue the algorithm by removing
vertices of type 1 with degree $<2$ and vertices of type 2 with degree
$<3$ to get the $2.5$-core.
\end{itemize}
In the case of (iii), if the graph induced by $a$ and the nodes
removed during this second phase is a tree, then it follows that $a$
is part of the $2.5+1.5$-core. Note that as long as the number of
nodes removed during this second phase is $o(n^{1/3})$, the graph
induced by $a$ and the removed nodes is a tree w.h.p.: for each node
the probability to connect to one of the $o(n^{1/3})$ nodes is
$o(n^{-2/3} \log n)$, and by a union bound over all nodes the desired
result follows. 

We clearly have
\begin{eqnarray}
\label{eq:upper}\Pr{(a \mbox{ in 2.5+1.5-core}) } \leq \Pr{(a \mbox{ has degree at least } t(a) \mbox{ in } C(a)) }, 
\end{eqnarray}
where $t(a)\in \{1,2\}$ is the type of $a$.
Hence, we need to compute the probability that $a$ has at least
$1$ neighbor in $C(a)$ in the case of being of type $1$ (at least $2$
neighbors in $C(a)$ in the case of being of type $2$). When changing
only one vertex $a$ to be always heavy, the functions
$h(x),h_1(x),h_2(x)$ change only by an additive $O(\log n/n)$, and
thus $\tilde{c}(q)$ and $\tilde{\xi}(q)$ also change by at most an
additive $o(1)$. Hence, for $c > \tilde{c}(q)$, as in the previous
proof, for the stopping time $\tau$ we still have $\tau \sim
-\log(\tilde{\xi}/c)$. Therefore, we have to compute the probability
that at time $\tau$, $a$ has at least $1$ neighbor in $C(a)$ in the case
of being of type $1$ (at least $2$ in the case of being of type
$2$). Note that since $a$ is heavy, the probability for each halfedge
incident to $a$ to be alive at time $t$ is $e^{-t}$. Since $a$ is
chosen uniformly at random, we have
\begin{eqnarray*}
  &\Pr{(a \mbox{ has degree at least }  t(a) \mbox{ in } C(a))}\\
  =&(1+o(1)) \sum_{k \geq 0} \left(\frac{e^{-c}c^k}{k!} \left( (1-q)(1-(1-e^{-\tau})^k)+q(1-(1-e^{-\tau})^k-ke^{-\tau}(1-e^{-\tau})^{k-1}) \right)\right) \\
  =&(1+o(1))\left(1-e^{-\tilde{\xi}}-q \tilde{\xi} e^{-\tilde{\xi}}\right).
\end{eqnarray*}
We now prove (\ref{eq:cheb}).
Consider two special vertices $a$ and
$b$, chosen uniformly at random, and consider them both heavy. By the
same reasoning as above, the functions $\tilde{c}(q), \tilde{\xi}(q),
h(x), h_1(x), h_2(x)$ change only by additive terms of $o(1).$ 
Again, note that in order for both $a$ and $b$ to be in the
$2.5+1.5$-core, it is necessary but not sufficient for both $a$ and
$b$ to have at time $\tau$ still $1$ incident edge in the case of
being of type $1$ ($2$ incident edges in the case of being of type
$2$).
Hence,
\begin{eqnarray*}
  &\Pr{(a \mbox{ and }b \mbox{ in 2.5+1.5-core}) }\\
  \leq &(1+o(1))\left(1-e^{-\tilde{\xi}}-q \tilde{\xi} e^{-\tilde{\xi}}\right)^2.
\end{eqnarray*}

We now prove that (\ref{eq:upper}) is tight.
First, we compute the degree distribution in $C(a)$.
We use the same notation as in the proof of Theorem \ref{theo_core}.
It follows from Lemma 4.4 in \cite{Janson} (see also Lemma 12 in
\cite{lelarge_diff}) that we have
\begin{eqnarray*}
\frac{U^1_r(t)}{n} &=& (1-q)\frac{\left(c
    e^{-t}\right)^re^{-c e^{-t}}}{r!} +o(1),\quad r\geq 2\\
\frac{U^2_r(t)}{n} &=& q\frac{\left(c
    e^{-t}\right)^re^{-c e^{-t}}}{r!} +o(1),\quad r\geq 3.
\end{eqnarray*}
In particular, at $\tau=-\log\left(\tilde{\xi}/c \right)+o(1)$, the
stopping time of the process, we have
\begin{eqnarray*}
  \frac{U^1_r(\tau)}{n} &=& (1-q)\frac{\tilde{\xi}^re^{-\tilde{\xi}}}{r!} +o(1),\quad r\geq 2\\
  \frac{U^2_r(\tau)}{n} &=& q\frac{\tilde{\xi}^re^{-\tilde{\xi}}}{r!} +o(1),\quad r\geq 3.
\end{eqnarray*}
As computed above, the total number of half-edges is
\[
  \sum_{r \geq 2}  rU^1_r(\tau)+ \sum_{r \geq 3}rU^2_r(\tau) = 
n h(\tilde{\xi}/c)+o(n)
=n\frac{\tilde{\xi}^2}{c}+o(n).
\]
Hence, when choosing a ball uniformly at random among the balls in
  bins corresponding to heavy vertices, the probability to pick a
node of type $1$ with degree $2$ is
\begin{eqnarray*}
  p^1_2 = \frac{2U^1_2(\tau)}{\sum_r rU^1_r(\tau)+rU^2_r(\tau)}=(1-q) e^{-\tilde{\xi}}c+o(1),
\end{eqnarray*}
and similarly, the probability to pick a node of type $2$ with degree
$3$ is
\begin{eqnarray*}
  p^2_3 = \frac{3U^2_3(\tau)}{\sum_r rU^1_r(\tau)+rU^2_r(\tau)}= q \frac{\tilde{\xi}c}{2}e^{-\tilde{\xi}}+o(1).
\end{eqnarray*}

Consider a new node of type $t\in\{1,2\}$, picked up during the second
phase of the algorithm in case (iii); as long as the neighborhood of
$a$ explored in this second phase is a tree, this new node has at
least $t+1$ half-edges, since it is heavy. If it has exactly $t+1$
half-edges, it becomes light after removal of one half-edge and the
algorithm continues: a type $1$ node then induces one more half-edge
to remove, and a type $2$ node two more half-edges. So to show that
this branching exploration process is $o(n^{1/3})$, it suffices to
prove that it is subcritical, that is,
\begin{eqnarray*}
2p^2_3+p^1_2<1,
\end{eqnarray*}
and this will imply that (\ref{eq:upper}) is indeed tight.
We have
\begin{eqnarray*}
2p^2_3+p^1_2 =q \tilde{\xi}ce^{-\tilde{\xi}} +(1-q) e^{-\tilde{\xi}}c+o(1),
\end{eqnarray*}
and we will now show that this is indeed less than $1$ for $c >
  \tilde{c}(q)$. 
Recall that for $c>\tilde{c}(q)$ (the case of interest here), 
$\tilde{\xi}(c,q)>0$ is the largest solution of the equation
\begin{equation}
c=\psi(\xi;q) \quad \mbox{with} \quad \psi(\xi;q)=\frac{\xi}{1-e^{-\xi}-q\xi 
e^{-\xi}}.
\label{eq:xitilde}
\end{equation}
Then, necessarily at the point $(\tilde{\xi}(c,q),q)$:
\begin{equation}\label{eq:derivative}
\frac{\partial \psi}{\partial \xi}=\frac{e^\xi (-q\xi^2-\xi+e^{\xi}-1)}{(q\xi-e^{\xi}+1)^2}\geq 0,
\end{equation}
as otherwise \eqref{eq:xitilde} would have a larger solution (note that $\psi(\xi;q) \to \infty$ as $\xi \to \infty$).
Furthermore, if $q\leq 1/2$, $\partial \psi/\partial \xi$ has no
strictly positive root, and in this case, for $c>\tilde{c}(q)$, $\frac{\partial
  \psi}{\partial \xi}>0$ at the point $(\tilde{\xi}(c,q),q)$.  Now, if
$q>1/2$, note that $\partial \psi/\partial \xi$ has a single strictly positive root:  indeed, for $\xi > 0$, $\partial \psi/\partial \xi (\xi)=0$ iff $g(\xi):=e^{\xi}-1-q\xi^2-\xi=0$. Now, $g'(\xi)=e^{\xi}-1-2q\xi$ and $g''(\xi)=e^{\xi}-2q$, and hence $g(\xi)$ is convex for $\xi \ge \log(2q) > 0$ and concave otherwise. Also, $g(0)=g'(0)=0$, and since $g''(0) < 0$, the function is first decreasing, and then increasing with $g(\xi) \to \infty$ as $\xi \to \infty$, therefore passing exactly once by $0$, giving the single strictly positive root.
This root is a minimum of $\psi$ and equal to
$\tilde{\xi}(\tilde{c}(q),q)$: recall the definition of $\tilde{c}(q)$
\[
\tilde{c}(q)=\inf_{\xi > 0} \psi(\xi;q).
\]
Hence for $c>\tilde{c}(q)$, $\partial \psi/\partial \xi$ cannot vanish at the
point $(\tilde{\xi}(c,q),q)$, thus it is strictly positive. 
Using \eqref{eq:derivative}, this is equivalent to 
\[
e^{-\tilde{\xi}}\left(1+\tilde{\xi}+q\tilde{\xi}^2\right)<1 
\]
and therefore
\[  
\tilde{\xi} \left(q e^{-\tilde{\xi}}\tilde{\xi}+(1-q)e^{-\tilde{\xi}} \right)
<(1-q)(1-e^{-\tilde{\xi}})+q(1-e^{-\tilde{\xi}}-\tilde{\xi}e^{-\tilde{\xi}}).
\]
Thus, the expression $2p^2_3+p^1_2$ at $c>\tilde{c}(q)$ is
\[
\frac{\tilde{\xi}
  e^{-\tilde{\xi}}\left(q\tilde{\xi}+(1-q)\right)}{(1-q)(1-e^{-\tilde{\xi}})+q(1-e^{-\tilde{\xi}}-\tilde{\xi}e^{-\tilde{\xi}})}<1,
\]
 as desired.
  
The result follows.

\section{Proof of Lemma~\ref{lemma_important}}\label{sec:lemma_important}

In order to prove Lemma~\ref{lemma_important}, we first need to prove the following auxiliary lemma:
\begin{lemma}\label{lem:lagrange}
  Let $\alpha \leq 1$. Let $r,s, t=\alpha s \in \mathbb{N}$ and let
  $s_1,\ldots,s_r \in \mathbb{N}$ such that $\sum_{i=1}^r s_i=\alpha
  s$ and $\sum_{i=1}^r i s_i=s$. There exists $C > 0$ such that for
  any $r > 0$, and any $\alpha$ we have
$$
\prod_{i=1}^r \frac{1}{s_i^{s_i}} \leq C^{\alpha s}/(\alpha^2 s)^{\alpha s}.
$$
\end{lemma}
\begin{proof}
It is sufficient to prove
\[
\sum_{i=1}^r s_i\log s_i \geq \alpha s\log\alpha s +\alpha s \log
\alpha - \alpha s \log C
\]
We first normalize $s_i$. Writing $\tilde{s}_i=s_i/(\alpha s)$, the
constraints are
\begin{eqnarray*}
\sum_{i=1}^r \tilde{s}_i =1 &;& \sum_{i=1}^r i\tilde{s}_i = \frac{1}{\alpha}\\
\end{eqnarray*}
We will show that we can find $C > 0$ independent of $r$ and $\alpha$ such that
\begin{equation}
\sum_{i=1}^r \tilde{s}_i\log \tilde{s}_i \geq  \log \alpha - \log C
\label{eq:toprove}
\end{equation}
To simplify notation, we write simply $s_i$ from now on.

For a fixed $r$, we maximize $\sum_{i=1}^r - s_i \log s_i$ subject to
$\sum_{i=1}^r s_i=1$ and $\sum_{i=1}^r i s_i=1/\alpha$. Applying
Lagrange multipliers gives the system of equations
$$
\log s_i + 1+\lambda_1+i \lambda_2=0, \hspace{2cm} i=1,\ldots,r,
$$
and thus an optimal solution has to satisfy $s_i=a b^i$ for some $a,b
\in \mathbb{R}$.  It is enough to show
\eqref{eq:toprove} for all $s_i$ of this form.\\
The two constraints translate into the two following equations,
repeatedly used in the following:
\begin{eqnarray*}
 a&=& \frac{1-b}{b(1-b^r)} \\
 \frac{1-(r+1)b^r+rb^{r+1}}{(1-b)(1-b^r)} &=& \frac{1}{\alpha}
\end{eqnarray*}
We distinguish five cases.\\
{\bf Case 1:} there exists $K > 0$ such that $1/\alpha\geq r/K$.\\
In this case, we may ignore the constraint $\sum_{i=1}^r i
s_i=1/\alpha$. Clearly, $\sum_{i=1}^r s_i\log s_i$ is minimized by the
uniform distribution $s_i=1/r$.  Hence
\[
\sum_{i=1}^r s_i\log s_i \geq -\log r  \geq \log \alpha -\log K, 
\]
and we are done.\\
{\bf Case 2:} $r<R_0$, where $R_0$ is some large enough constant such that $(r-1)(1/2)^{r+1}<0.5$ for all $r\geq R_0$. \\
Reasoning as in the previous case, we obtain
\[
\sum_{i=1}^r s_i\log s_i \geq -\log R_0,
\]
and we are done as well.\\
{\bf Case 3:} $b\geq 1$.\\
Note that we always have 
\[
\sum_{i=\lfloor r/2\rfloor }^r s_i \leq \frac{\sum_{i=1}^r is_i}{\lfloor r/2\rfloor}.
\]
If $b\geq 1$, $s_i$ is increasing, thus the left hand side is larger
than $1/2$. Hence $(1/\alpha)=\sum_{i=1}^r is_i \geq \lfloor
r/2\rfloor/2 \geq r/5$. We may choose $K=5$ and
apply Case 1.\\
{\bf Case 4:} $b\leq 1/2$, and $r\geq R_0$.\\
We have
\[
a=\frac{1}{b(1+b+\ldots+b^{r-1})} =  \frac{1}{b(1+c(b) b)} 
\]
with $1\leq c(b)\leq 2$. Using $\sum_{i=1}^r s_i\log s_i=\sum_{i=1}^r
s_i (\log a + i \log b)=\log a + \frac{1}{\alpha} \log b$ we obtain
\begin{equation}
  \sum_{i=1}^r s_i\log s_i= \log a +\frac{1}{\alpha} \log b = (1/\alpha -1)\log b -\log (1+c(b) b) \geq (\frac{1}{\alpha} -1)\log b -\log 2.
\label{eq:step5}
\end{equation}
Using that $(1-b)(1-b^r) \geq \frac14$, we get
\begin{eqnarray}
  \frac{1}{\alpha} &=& \frac{1-(r+1)b^r+rb^{r+1}}{(1-b)(1-b^r)} \nonumber \\
  \frac{1}{\alpha}-1 &=& \frac{b-rb^r+(r-1)b^{r+1}}{(1-b)(1-b^r)} \nonumber\\ 
  &\leq & 4\left( b+ (r-1)b^{r+1}\right) \nonumber\\
  &\leq& 4b(1+0.5) =6b\nonumber, 
\end{eqnarray}
where we used that for $R \geq R_0$, we have $(r-1)b^{r+1}\leq
(r-1)(1/2)^{r+1} \leq 0.5$.  Thus,
$$
 \left(\frac{1}{\alpha}-1\right) \log b \geq 6 b\log b \geq C \nonumber,
$$
since $6b\log b$ is for $b\in [0,~1/2]$ bounded below by an absolute
constant.
Combining this last estimate with \eqref{eq:step5}, the desired inequality holds.\\
{\bf Case 5:} $1/2 \leq b<1$ and $1/\alpha\leq r/K$.\\
Using $\log a = \log(1-b)-\log b -\log (1-b^r) \geq \log(1-b)$, we
obtain
\[
\sum_{i=1}^r s_i\log s_i = \log a +\frac{1}{\alpha}\log b \geq \log
(1-b) +\frac{1}{\alpha}\log b.
\]
We will show that $\log b/\alpha$ is bounded from below by some
constant, and that $\log(1-b)\geq \log\alpha +C$.  First note that we
have the following relation:
\begin{equation}
\frac{1-b}{\alpha} = 1-r\frac{b^r(1-b)}{1-b^r} 
\label{eq:1mb}
\end{equation}
Consider first $\log b/\alpha$. If $\alpha \geq 1/2$, then $\log
b/\alpha\geq 2\log b\geq -2\log 2$, and we have the desired
bound. Assume then $\alpha < 1/2$.  From \eqref{eq:1mb}, we have
$(1-b)/\alpha \leq 1$, and thus, $\log b\geq \log (1-\alpha)$. Also,
since $\alpha<1/2$, $\log(1-\alpha)\geq -2\alpha$.  Hence,
\[
\frac{\log b}{\alpha}\geq -2,
\]
and we are done with this term. Let us turn to the $\log(1-b)$
term. For any $i$, we have $s_i\geq ab^r$, and hence
\[
\frac{r}{K}\geq \frac{1}{\alpha}=\sum_{i=1}^r is_i \geq ab^r \frac{r(r+1)}{2}
\]
Thus 
\[
\frac{2b}{K}\geq (r+1) \frac{b^r(1-b)}{1-b^r},
\]
and since $br/(r+1) \leq 1$, also
\[
\frac{2}{K}\geq r \frac{b^r(1-b)}{1-b^r}
\]
Inserting this into \eqref{eq:1mb}, we have
\[
\frac{1-b}{\alpha} \geq 1-\frac{2}{K}
\]

Thus $\log(1-b)\geq \log\alpha +C$, and we are done.
\end{proof}

We need the following two lemmas.
\begin{lemma}\label{nocycle}\cite{Bollobas}[Corollary 5.8]\\
  Let $G \in \mathcal{G}(n,p)$ with $p=c/n$ and $0 < c < 1$. The
  following holds a.a.s.:
\begin{itemize}
\item $G$ contains only trees and unicyclic components.
\item The number of vertices in unicyclic components is at most
  $\omega_n$, for some arbitrarily slowly growing function $\omega_n$.
\end{itemize}
\end{lemma}

The following lemma can easily be derived from Corollary 5.11 and
Theorem 5.5 of~\cite{Bollobas}, by extending it to non-constant values of
$k$:
\begin{lemma}\label{treesizes}
  Let $G \in \mathcal{G}(n,p)$ with $p=c/n$ and $0 < c < 1$.  \ Let
  $T_k$ be the number of trees of size $k$ in $G$. Let $\omega_n$ be
  an arbitrarily slowly growing function with $n$. The following holds
  a.a.s.:
\begin{itemize}
\item $T_k=0$ for $k =\omega(\log n)$. 
\item  $T_1=ne^{-c}(1+o(1))$.
\item For any $2 \leq k =O(\log n)$ with
  $\frac{nk^{k-2}}{k!}c^{k-1}e^{-kc}\geq \omega_n$, $T_k=
  \frac{nk^{k-2}}{k!}c^{k-1}e^{-kc}(1+o(1))$.
\end{itemize}
\end{lemma}

We need one more lemma, in the spirit of Theorem 5.10
of~\cite{Bollobas}.
\begin{lemma}\label{logntrees}
  Let $G \in \mathcal{G}(n,p)$ with $p=c/n$ and $0 < c < 1$.  Let
  $\eta > 0$ be a sufficiently small constant.  Let $\omega_n$ be any
  function growing with $n$ arbitrarily slowly, and let $k_0 :=
  \lfloor \nu \log n \rfloor$ be the smallest integer $k$ satisfying
  $\eta \omega_n \leq \mathbb{E}(T_k)=
  \frac{nk^{k-2}}{k!}c^{k-1}e^{-kc} \leq \omega_n$ (since $0 < c < 1$,
  such an integer must exist for sufficiently small $\eta$). Then there 
exists a constant $C>0$ such that a.a.s. $\sum_{k \geq
    k_0} T_k \leq C \mathbb{E}(T_{k_0})$. 
\end{lemma}
\begin{proof}
  Clearly, $\mathbb{E}(T_{k_0}) \leq \sum_{k \geq k_0}
  \mathbb{E}(T_{k})$. By Stirling's formula,
\begin{eqnarray*}
  \sum_{k \geq k_0} \mathbb{E}(T_{k}) &= & (1+o(1)) (ce^{1-c})^{k_0}\frac{n}{\sqrt{2\pi}c}\sum_{i \geq 0} \frac{(ce^{1-c})^i}{(i+k_0)^{2.5}} \\
  & \leq &
  (1+o(1))(ce^{1-c})^{k_0}\frac{n}{\sqrt{2\pi}c k_0^{2.5}}\sum_{i \geq 0}(ce^{1-c})^i \\
  &= &\mathbb{E}(T_{k_0}) \frac{1}{1-(ce^{1-c})^i}.
\end{eqnarray*}
Writing $T_k$ as a sum of indicator variables over all $k$-tuples of
vertices, we see that when considering two disjoint trees of size at
most $O(\log n)$, at most $O(\log n)$ non-edges incident to each
vertex of the second tree are exposed when given the first. Hence, the
probability of having no edge adjacent to any of the vertices changes
only by a factor $(1+o(1))$, and thus $\mathbb{E}(( \sum_{k \geq k_0}
T_k)^2)=2\mathbb{E}(\sum_{\ell > k \geq k_0} T_k
T_{\ell})+\mathbb{E}(\sum_{k \geq k_0}T_k^2)=(1+o(1))\sum_{k \geq k_0}
\mathbb{E}(T_k)\sum_{\ell \geq k_0}\mathbb{E}(T_{\ell})=(1+o(1))(
\sum_{k \geq k_0} \mathbb{E}(T_{k}))^2.$ By Chebyshev's inequality,
the result follows. 
\end{proof}

We turn now to the proof of Lemma~\ref{lemma_important}.

\begin{proof}
  For a subgraph of size $u$, let $n_1$ its number of vertices of type
  $1$ and $n_2$ its number of vertices of type $2$ (we do not explicitly
  refer to the size nor to the subgraph, since it is clear from the
  context).  Let $X_u$ denote the number of subgraphs of size $u \leq
  \alpha n$ with more than $n_1+2n_2$ edges. Our goal is to show that
  for a randomly chosen graph $G \in \mathcal{G}(n,p)$ we have
  $\sum_{u \leq \alpha n} X_u=0$ a.a.s.  To simplify the notation, we
  set $r=1-q$. Let $\omega_n$ denote a function tending to infinity 
arbitrarily slowly, as $n \to \infty$.

We start with the relatively easy cases where $u$ is small enough, or $n_2$ 
large enough.\\

\emph{Small $u$: $u=o(\log n /\log \log n)$}\\
  First note that the expected number of subgraphs of size 
$u =o(\log n/\log \log n)$ with at least $n_1+2n_2+1$ edges is at most
\begin{align*}
  & \sum_{n_1=0}^{\log n/(\omega_n \log \log n)}\sum_{n_2=0}^{\log n/(\omega_n \log \log n)} \binom{rn(1+o(1))}{n_1}\binom{qn(1+o(1))}{n_2}\binom{\binom{n_1+n_2}{2}}{n_1+2n_2+1}p^{n_1+2n_2+1}  \\
  \leq &\sum_{n_1=0}^{\log n/(\omega_n \log \log n)}\sum_{n_2=0}^{\log n/(\omega_n \log \log n)}\left(\frac{rne}{n_1}\right)^{n_1}\left(\frac{qne}{n_2}\right)^{n_2}\left(\frac{ce(n_1+n_2)^2}{2n(n_1+2n_2+1)}\right)^{n_1+2n_2+1} \\
  \leq & \sum_{n_1=0}^{\log n/(\omega_n \log \log n)}\sum_{n_2=0}^{\log n/(\omega_n \log \log n)} \left(\frac{1}{n}\right)^{n_2+1}
  \left(O(n_1+n_2)\right)^{n_1+2n_2+1} \\
  \leq & \frac{\log^2 n}{n} (O(\log n))^{(2\log n/\omega_n \log \log n)}=o(1),
\end{align*}
and thus a.a.s. there is no such subgraph. \\

\emph{Large $n_2$: $n_2 > \xi n_1$}\\
 Also, a subgraph with more
than $n_1+2n_2+1$ edges and total size at most $\alpha n$ cannot
exist, if the density of nodes of type $2$ is too big: more precisely,
if $n_2 > \xi n_1$ for some constant $\xi > 0$, then by Proposition~4
of~\cite{Theran} applied with $a=\frac{1+2\xi}{1+\xi} > 1$, there
exists a constant $t(a,c)=(\frac{2a}{c})^{a/(a-1)}e^{-(a+1)/(a-1)}$,
such that a.a.s. $G(n,c/n)$ has no subgraph with $n_1+2n_2$ edges of
size at most $t(a,c)n$.\\

\emph{Remaining cases: $u=\Omega(\log n/\log \log n)$ and
$n_2 \leq n_1/K$ }\\ 
We may thus assume that we are dealing with the remaining cases of
$u=\Omega(\log n/\log \log n)$ and $n_2 \leq n_1/K$. We drop the
condition $u \leq \alpha n$ from now on.

Let $K'=K'(c-1/r)$ the constant coming from part~2 of
Lemma~\ref{logntrees} so that a.a.s. all trees are of size at most $K'
\log n$, and let $K=K(K')$ be a sufficiently large constant (in fact,
$K$ depends on $q$ and $c$, and it is the largest
constant appearing throughout this proof; in particular, we choose $K$
to be such that it is also larger than the product of $1/rc$ and the
constants $C$ of the statement of Lemma~\ref{lem:lagrange} and the
constant $C$ of the statement of Lemma~\ref{logntrees}).

We need one more observation: when counting all components $X_u$
with more than $n_1+2n_2$ edges, let us first choose the components
from the subgraph induced by the $n_1$ vertices. We may assume that
each of these components is taken entirely or not at all: indeed,
since $p =c/n$ with $c < 1/r$, and since the probability for an edge
to be present in the subgraph induced by all $rn(1+o(1))$ vertices of
type $1$ is $p=c/n=\frac{cr}{rn} < \frac{1}{rn}$, for the subgraph
induced by all $rn(1+o(1))$ vertices of type $1$, Lemma~\ref{nocycle}
applies (with $cr$ playing the role of $c$, and $rn$ playing the role
of $n$), and thus this subgraph a.a.s. contains only trees and
unicyclic components. If there were now a subgraph with $n_1$ vertices
of type $1$ and $n_2$ vertices of type $2$ and more than $n_1+2n_2$
edges, where some components are partially taken, we could complete
the tree components and unicyclic components. We add at least the same
number of edges as vertices, and the resulting subgraph would have
$n_1'$ vertices of type $1$, $n_2$ vertices of type $2$ (clearly still
satisfying $n_2 \leq n_1'/K$), and it would still have more than
$n_1'+2n_2$ edges. Thus we will from now on count only subgraphs where
all components of the subgraph induced by the vertices of type $1$ are
entirely or not at all taken.

Denote by $T_i$ the number of trees of size $i$ of this subgraph
induced by the vertices of type $1$. By Lemma~\ref{treesizes}, a.a.s.,
$T_1=rne^{-rc}(1+o(1)) \leq (rnrce^{1-rc})\frac{1}{rc}$, and using
Stirling's formula, for each $2 \leq i = O(\log n)$ satisfying
$\frac{rni^{i-2}}{i!}(rc)^{i-1}e^{-irc} \geq \omega_n$, a.a.s.,
$T_i=(1+o(1))\frac{rni^{i-2}}{i!}(rc)^{i-1}e^{-irc} \leq
rn(rce^{1-rc})^i \frac{1}{rc}$.  For the remaining number of trees of
size $i=\Theta(\log n)$, let $i_0:=\lfloor \nu \log n\rfloor$
throughout the proof be the smallest integer $i$ satisfying
$$\eta \omega_n \leq \mathbb{E}(T_{i})
= \frac{rni^{i-2}}{i!}(rc)^{i-1}e^{-irc}=
(1+o(1))\frac{rn(rce^{1-rc})^i}{\sqrt{2\pi}rc i^{2.5}} \leq \omega_n
$$ for some sufficiently small $\eta > 0$ (as remarked before, such $i_0$ exists for small enough $\eta > 0$). By the proof of Lemma~\ref{logntrees}, the number of trees of size at least $i_0$ is a.a.s. at most $(1+o(1))\frac{1}{1-cr e^{1-cr}} \mathbb{E}(T_{i_0}) \leq  (1+o(1))\frac{1}{1-cre^{1-cr}}rn(rce^{1-rc})^{i_0}$. Finally, once more by Lemma~\ref{treesizes}, the total number of trees of size $i =\omega(\log n)$ is $0$ a.a.s.

Consider now any graph on $n$ vertices, for which the subgraph induced
by $(1+o(1))rn$ vertices (that will then correspond to vertices of
type $1$) is deterministically given: it consists only of trees and
unicyclic components, the number of trees of size $i < i_0$ is at most
$rn(rce^{1-rc})^i \frac{1}{rc}$, the total number of trees of size
$i=\Theta(\log n)$ for $i \geq i_0$ is at most
$(1+o(1))\frac{1}{1-cre^{1-cr}} rn(rce^{1-rc})^{i_0} $, the number of
trees of size $\omega(\log n)$ is $0$, and the number of vertices in
unicyclic components is at most $\omega_n$. We will below show that
starting with any such graph, when exposing the random edges between
the $qn$ vertices of type $2$ and edges between type $1$ and type $2$
(as before, each such edge being present with probability $p$), with
probability $1+o(1)$ in the whole graph there are no subgraphs $X_u$,
$u \leq \alpha n$ with more than $n_1+2n_2$ edges.  The lemma will
then follow, since the randomly chosen graph $G \in \mathcal{G}(n,p)$,
as shown above, a.a.s. satisfies these properties.

\par
It remains now to show that any graph on $(1+o(1))rn$ vertices (of
type $1$) with the above mentioned properties has a.a.s. no subgraphs
with more than $n_1+2n_2$ edges. By definition, we can bound the
number of such components of each size in the subgraph induced by the
vertices of type $1$. We then have to combine them with the choices
for the $n_2$ vertices. If there are $t$ tree components in the
subgraph induced by the vertices of type $1$, the number of additional
edges needed to surpass $n_1+2n_2$ is $2n_2+t+1$. Call tree components
of size $i < i_0$ to be \emph{small}, and call trees of size $i_0 \leq
i=\Theta(\log n)$ to be \emph{medium}. For any subgraph with
$n_1=\Omega(\log n/\log \log n)$ vertices of type $1$, write
$n_1=n_t+z$, where $n_t$ is the number of vertices belonging to tree
components and $z$ is the number of vertices belonging to unicyclic
components. Note that $z \leq \omega_n$ and that by our assumption on
$n_1$, $z=o(n_1)$. Next, for small trees of size $i < i_0$, let $s_i$
be the number of tree components of size $i$ in the subgraph of the
$n_1$ vertices, and denote by $s_{i_0}$ the number of medium tree
components (of size at least $i_0$) in the subgraph of the $n_1$
vertices. Let $\sum_{i=1}^{i_0} i s_i=s$, and let $\sum_{i=1}^{i_0}
s_i = \alpha s$ for $\alpha \leq 1$. Note that $s$ is therefore a
lower bound on the number of vertices in trees in the subgraph of the
$n_1$ vertices; also, note that all medium sized trees are of size
$\Theta(\log n)$, and therefore $s \geq n_1/K'$ and still $z=o(s)$.)
Define now $X_u^{s,\alpha
  s,n_1,n_2}$ the number of subgraphs $X_u$ with $n_1$ vertices of
type $1$, $n_2$ vertices of type $2$, satisfying $\sum_{i=1}^{i_0} i
s_i=s$, and $\sum_{i=1}^{i_0} s_i = \alpha s$, furthermore having at
most $\omega_n$ vertices in unicyclic components and having in total
more than $n_1+2n_2$ edges (the previously imposed restrictions $s
\geq n_1/K'$, $n_2 \leq n_1/K$ and $n_1=\Omega(\log n/\log \log n)$
still hold). We have for some large constants $C,C',C'' > 0$ (whose
values might change from line to line)
\begin{align*}
  \mathbb{E}(X_u^{s,\alpha s,n_1,n_2}) \leq &   \sum_{\sum s_i =\alpha s, \sum_i i s_i =s}  \left( 2^{\omega_n} \prod_{i} \binom{rn(rce^{1-rc})^i C}{s_i}\binom{qn(1+o(1))}{n_2}\binom{n_1n_2+\binom{n_2}{2}}{2n_2+\sum s_i} p^{2n_2+\sum s_i+1} \right) \\
  \leq & 2^{\omega_n}(C'n)^{\alpha s}(rce^{1-rc})^s
  \left(\frac{qen(1+o(1))}{n_2}\right)^{n_2} \left(\frac{\left( C''
        sn_2+\frac12n_2^2\right)ec}{(2n_2+\alpha
      s)n}\right)^{2n_2+\alpha s} p \prod_i \frac{1}{s_i^{s_i}}.
\end{align*}
By Lemma~\ref{lem:lagrange}, $\prod_i \frac{1}{s_i^{s_i}} \leq
\left(\frac{C}{\alpha^2 s}\right)^{\alpha s}$ for some $C > 0$, and
writing $n_2=\beta \alpha s$ (note that since $n_2 \leq n_1/K$ and $s \geq n_1/K'$, we
still have $\beta \alpha \leq K'/K$, which is still sufficiently small
for large enough $K=K(K')$), we have (again for large enough $C,C' >
0$)
\begin{align*}
2^{-\omega_n}  \mathbb{E}(X_u^{s,\alpha s,n_1,n_2}) \leq & \left( \frac{C'^{\alpha}rce^{1-rc} (\beta \alpha s^2)^{\alpha}}{(\alpha^2 s)^{\alpha}(2\beta\alpha s+\alpha s)^{\alpha})} \right)^s \left( \frac{C\left( \beta \alpha s^2 \right)^2}{\beta \alpha s (2\beta\alpha s+\alpha s)^2 n} \right)^{\beta \alpha s} \\
  \leq & \left(
    \frac{C'^{\alpha}rce^{1-rc}\beta^{\alpha}}{\alpha^{2\alpha}
      (1+2\beta)^{\alpha}} \frac{(C\beta s)^{\beta
        \alpha}}{(1+2\beta)^{2\beta\alpha} (\alpha n)^{\beta \alpha}}
  \right)^s.
\end{align*}
We distinguish now three cases. If $\beta \geq K'/K^{3/4}$, then
$\alpha \leq 1/K^{1/4}$. Then, for some $C'' > 0$, we have
\begin{align*}
 2^{-\omega_n} \mathbb{E}(X_u^{s,\alpha s,n_1,n_2}) \leq &
  \left(\frac{(rce^{1-rc})C''^{\alpha} (Cs)^{\beta
        \alpha}}{\alpha^{2\alpha} \beta^{\beta \alpha} (\alpha
      n)^{\beta \alpha}}\right)^s.
\end{align*}
The base of the last expression is clearly monotone increasing in $s$,
and so we may plug in our upper bound on $s \leq u \leq
n$. Considering the base only and taking logarithms, we obtain
$$
\log(rce^{1-rc}) -2\alpha\log \alpha - \beta \alpha \log (\beta\alpha)
+\alpha \log C''+\beta \alpha \log C.
$$
Note that $x \log x \to 0$ as $x \to 0$, and thus, if $K$ and thus
also $K^{1/4}$ is sufficiently large, both $\alpha$ and $\beta \alpha$
are sufficiently small. Then the first term dominates in absolute
value, and since $rce^{1-rc} < 1$, the expression is negative. 

 If $\beta < K'/K^{3/4}$ and $\alpha \leq 1/K^{1/4}$, then for some $C'' >
0$
\begin{align*}
2^{-\omega_n}  \mathbb{E}(X_u^{s,\alpha s,n_1,n_2}) \leq &
  \left(\frac{(rce^{1-rc})C''^{\alpha} s^{\beta \alpha}\beta^{\beta
        \alpha+\alpha}}{\alpha^{2\alpha+\beta \alpha} n^{\beta
        \alpha}}\right)^s.
\end{align*}
Reasoning as before, we obtain
$$
\log(rce^{1-rc}) + (\beta \alpha + \alpha)\log \beta - (2\alpha+\beta
\alpha) \log \alpha + \alpha \log C''.
$$
The second term is negative, and among the others, for $K$ large
enough, the first term dominates them in absolute value, and hence the
expression is negative.

Finally, if $\alpha > 1/K^{1/4}$ and hence $\beta < K'/K^{3/4}$, for
some $C'' > 0$ we have
\begin{align*}
 2^{-\omega_n} \mathbb{E}(X_u^{s,\alpha s, n_1, n_2}) \leq &
  \left(\frac{C''^{\alpha}rce^{1-rc} s^{\beta \alpha}\beta^{\beta
        \alpha+\alpha}}{\alpha^{2\alpha+\beta \alpha} n^{\beta
        \alpha}}\right)^s.
\end{align*}
As before, we obtain
\begin{align*}
  \log(rce^{1-rc}) + (\beta \alpha + \alpha)\log \beta - (2\alpha+\beta \alpha) \log \alpha+\alpha \log C'' \\
  = \log(rce^{1-rc}) + \beta \alpha \log (\beta/\alpha) + \alpha \log
  (C''\beta/\alpha^2).
\end{align*}
Once more for $K$ large enough, $\beta < \alpha$, and the second term
is negative. Also, for $K$ large enough, $C'' < K^{1/4} / K'$, and
thus $C''\beta < 1/K^{1/2}$, and therefore $C''\beta < \alpha^2$, and
the last expression is negative as well.
In all cases, since we have $s=\Theta(u)$,
$$2^{-\omega_n} \mathbb{E}(X_u^{s,\alpha s, n_1, n_2}) \leq \rho^s \leq \rho^{Cu},
$$
for some $0 < \rho < 1$ and some absolute constant $C > 0$. Clearly, 
$$\mathbb{E}(X_u)=\sum_{s,\alpha s,n_1,n_2}\mathbb{E}(X_u^{s,\alpha s, n_1, n_2}) \leq 
2^{\omega_n} u^4 \rho^{Cu}.$$ 
Finally, $$ \sum_{\log \log n \leq u \leq \epsilon n} \mathbb{E}(X_u)
\leq 2^{\omega_n} \sum_{\log \log n \leq u \leq \epsilon n} u^4 \rho^{Cu}=o(1).$$
By Markov's inequality the lemma follows.
\end{proof}

\section*{Acknowledgements} We thank Louis Theran for helpful comments regarding Remark~\ref{rem:Theran}.


\end{document}